\documentclass{amsart}%
\usepackage{amssymb}
\usepackage{amsfonts}
\usepackage{geometry}
\usepackage{graphicx}
\usepackage{amsmath}
\usepackage[notref,notcite,final]{showkeys}
\usepackage[colorlinks=true,linkcolor=black,anchorcolor=black,citecolor=black,filecolor=black,menucolor=black,runcolor=black,urlcolor=black]%
{hyperref}%
\providecommand{\U}[1]{\protect\rule{.1in}{.1in}}
\newtheorem{theorem}{Theorem}[section]
\theoremstyle{plain}

\newtheorem{corollary}[theorem]{Corollary}

\newtheorem{definition}[theorem]{Definition}

\newtheorem{lemma}[theorem]{Lemma}

\newtheorem{proposition}[theorem]{Proposition}
\newtheorem{remark}[theorem]{Remark}

\numberwithin{equation}{section}
\def\Xint#1{\mathchoice
{\XXint\displaystyle\textstyle{#1}}{\XXint\textstyle\scriptstyle{#1}}{\XXint\scriptstyle\scriptscriptstyle{#1}}{\XXint\scriptscriptstyle\scriptscriptstyle{#1}}\!\int}
\def\XXint#1#2#3{{\setbox0=\hbox{$#1{#2#3}{\int}$ }
\vcenter{\hbox{$#2#3$ }}\kern-.6\wd0}}

\def\dashint{\Xint-}
\begin{document}
\title[Continuity for infinitely degenerate equations]{Regularity for infinitely degenerate inhomogeneous elliptic equations, Part
II: a lower half-Harnack inequality and continuity}
\author{Lyudmila Korobenko}
\address{Reed College\\
Portland, Oregon, USA}
\email{korobenko@reed.edu}
\author{Cristian Rios}
\address{University of Calgary\\
Calgary, Alberta, Canada}
\email{crios@ucalgary.ca}
\author{Eric Sawyer}
\address{McMaster University\\
Hamilton, Ontario, Canada}
\email{sawyer@mcmaster.ca}
\author{Ruipeng Shen}
\address{Center for Applied Mathematics\\
Tianjin University \\
Tianjin 300072, P.R.China}
\email{srpgow@163.com}
\thanks{Research supported by the Natural Sciences and Engineering Research Council of Canada}
\date{\today }

\begin{abstract}
We prove continuity of weak solutions to inhomogeneous quasilinear infinitely
degenerate elliptic equations%
\[
-\operatorname{div}\mathcal{A}\left(  x,u\right)  \nabla u=\phi_{0}%
-\operatorname{div}_{A}\vec{\phi}_{1},
\]
under abstract assumptions on the degenerate geometry: an Orlicz--Sobolev
inequality with a superradius, a $(1,1)$-Poincar\'{e} inequality, a growth
condition relating the superradius to the rate at which the measure of metric
balls decays, and a Dini condition on the right hand side. The proof combines
the $L^{\infty}$ bounds for negative powers of supersolutions obtained in Part
I with a Bombieri--Giusti lemma to obtain a lower half-Harnack inequality with
an explicit radius-dependent constant. A measure-to-pointwise lemma in the
spirit of De Giorgi then yields an oscillation estimate, which can be iterated
because the growth condition keeps the accumulated constants under control. No
upper half-Harnack inequality is needed.

As an application, for the geometries $f_{k,\sigma}(x_{1})=|x_{1}|^{(\ln
^{(k)}\frac{1}{|x_{1}|})^{\sigma}}$ in the plane we show that weak solutions
of $-\operatorname{div}\mathcal{A}(x,u)\nabla u=\phi_{0}-\partial_{x_{1}}%
\phi_{1}^{1}-f_{k,\sigma}(x_{1})\partial_{x_{2}}\phi_{1}^{2}$, with
$\mathcal{A}\approx\operatorname{diag}(1,f_{k,\sigma}^{2})$, are continuous
when $k=3$ and $0<\sigma<\frac{1}{2}$, or $k\geq4$ and $\sigma>0$. We also
give an application to the Monge--Amp\`{e}re equation $\det D^{2}%
u=f_{3,\sigma}(x)^{2}$ in the plane: convex $C^{1,1}$ solutions whose
restriction to the degenerate line is strictly convex are $C^{2}$.

\end{abstract}
\maketitle
\tableofcontents

\section{Introduction and main results\label{sec-intro}}

We consider divergence form quasilinear degenerate elliptic equations%
\begin{equation}
\mathcal{L}u\equiv-\operatorname{div}\mathcal{A}\left(  x,u\right)  \nabla
u=\phi_{0}-\operatorname{div}_{A}\vec{\phi}_{1}\qquad\text{in }\Omega
,\label{eq-quasi}%
\end{equation}
in a bounded domain $\Omega\subset\mathbb{R}^{n}$. Here $A(x)$ is a
nonnegative semidefinite symmetric matrix such that $\sqrt{A}$ is bounded and
Lipschitz in $\Omega$. The $A$-gradient and $A$-divergence are%
\begin{equation}
\nabla_{A}=\sqrt{A(x)}\,\nabla,\qquad\operatorname{div}_{A}\vec{V}%
=\operatorname{div}\left(  \sqrt{A(x)}\,\vec{V}\right)  .\label{def grad A}%
\end{equation}
The matrix $\mathcal{A}(x,z)$ is comparable to $A$ in the sense of quadratic
forms. That is, $\mathcal{A}(\cdot,z)\in\mathfrak{A}(A,\Lambda,\lambda)$
uniformly in $z\in\mathbb{R}$, where $\mathfrak{A}(A,\Lambda,\lambda)$ is the
class of nonnegative symmetric matrices $\tilde{A}(x)$ with%
\begin{equation}
\lambda\,\xi^{\mathrm{tr}}A(x)\xi\leq\xi^{\mathrm{tr}}\tilde{A}(x)\xi
\leq\Lambda\,\xi^{\mathrm{tr}}A(x)\xi,\qquad\text{a.e. }x\in\Omega,\ \xi
\in\mathbb{R}^{n},\label{struc_0}%
\end{equation}
for fixed constants $0<\lambda\leq\Lambda<\infty$. We assume that
$\mathcal{A}$ is a Carath\'{e}odory function (measurable in $x$, continuous in
$z$), so that $\tilde{A}(x)=\mathcal{A}(x,u(x))$ is measurable whenever $u $
is. If $u$ solves (\ref{eq-quasi}), then $u$ also solves the \emph{linear}
equation%
\begin{equation}
L_{\tilde{A}}u\equiv-\operatorname{div}\tilde{A}\nabla u=\phi_{0}%
-\operatorname{div}_{A}\vec{\phi}_{1},\qquad\tilde{A}(x)=\mathcal{A}%
(x,u(x))\in\mathfrak{A}(A,\Lambda,\lambda).\label{eq-linear}%
\end{equation}
All our estimates depend on $\tilde{A}$ only through $A,\lambda,\Lambda$. So
it suffices to work with linear equations (\ref{eq-linear}).

The degeneracy we have in mind is \emph{infinite}: $A$ may vanish to infinite
order. A model case in the plane is%
\begin{equation}
A(x_{1},x_{2})=\left[
\begin{array}
[c]{cc}%
1 & 0\\
0 & f(x_{1})^{2}%
\end{array}
\right]  ,\qquad f_{k,\sigma}(x_{1})=e^{-F_{k,\sigma}(|x_{1}|)},\quad
F_{k,\sigma}(r)=\left( \ln\frac{1}{r}\right) \left( \ln^{(k)}\frac{1}%
{r}\right) ^{\sigma},\label{f k sigma}%
\end{equation}
where $\ln^{(k)}$ denotes the $k$-fold iterated logarithm. In this case
$f_{k,\sigma}(x_{1})=|x_{1}|^{(\ln^{(k)}\frac{1}{|x_{1}|})^{\sigma}}$ vanishes
to infinite order at $x_{1}=0$. The associated Carnot--Carath\'{e}odory balls
are not doubling, so the classical De Giorgi--Nash--Moser theory does not apply.

\subsection{Main results}

The first main result is a geometric continuity theorem. Let $\Phi_{m}$ be the
Young function in (\ref{def Phi m}) below and $\Phi_{m}^{\ast}$ its conjugate.

\begin{theorem}
[geometric continuity]\label{th-geom-cont}Let $k\geq3$ and $\sigma>0$, with
$\sigma<\frac{1}{2}$ if $k=3$. Let $\Omega\subset\mathbb{R}^{2}$ be a bounded
domain. Let $f$ be an even Lipschitz function on $\mathbb{R}$, positive on
$\mathbb{R}\setminus\{0\}$, with $f=f_{k,\sigma}$ on $[-r_{\ast},r_{\ast}]$
for some $r_{\ast}>0$. Let $A=\operatorname{diag}(1,f(x_{1})^{2})$, and let
$\mathcal{A} $ be a Carath\'{e}odory matrix function with $\mathcal{A}%
(\cdot,z)\in\mathfrak{A}(A,\Lambda,\lambda)$ for all $z$. Fix $m>2$, with
$\sigma m(m-1)<1$ if $k=3$. Assume that

\begin{enumerate}
\item $\phi_{0}\in L_{\mathrm{loc}}^{\Phi_{m}^{\ast}}(\Omega)$;

\item $\vec{\phi}_{1}\in L_{\mathrm{loc}}^{\infty}(\Omega;\mathbb{R}^{2})$,
and for every compact $K\subset\Omega$ there is a nondecreasing $\eta
_{K}:(0,1]\rightarrow\lbrack0,\infty)$ with%
\begin{equation}
\int_{0}^{1}\eta_{K}(t)\,\frac{dt}{t}<\infty\qquad\text{and}\qquad|\vec{\phi
}_{1}(x)|\leq\eta_{K}(|x_{1}|)\quad\text{for a.e. }x\in K\text{ with }%
|x_{1}|\leq1.\label{drift Dini}%
\end{equation}
Then every weak solution $u\in W_{A}^{1,2}(\Omega)$ of (\ref{eq-quasi}) has a
representative that is continuous in $\Omega$.
\end{enumerate}
\end{theorem}

Since $m>2$ may be taken as close to $2$ as we wish, the restriction $\sigma
m(m-1)<1$ for $k=3$ allows every $0<\sigma<\frac{1}{2}$. The need for three
iterated logarithms (i.e. $k\geq3$), and the threshold $\sigma<\frac{1}{2}$
for $k=3$, are explained in Remark \ref{rem-three-logs}.

\begin{remark}
\label{rem-drift}Condition (\ref{drift Dini}) only restricts $\vec{\phi}_{1}$
near the degenerate line $\{x_{1}=0\}$. Away from it the equation is uniformly
elliptic, and bounded drifts are allowed. Near the line, some condition beyond
boundedness is needed for our method. In Remark \ref{rem-drift-counter} we
show that for $\vec{\phi}_{1}=(\mathbf{1}_{\{x_{1}>0\}},0)$ the size of the
drift, measured as in Definition \ref{def A admiss}, does not decay as the
radius of the ball tends to zero.
\end{remark}

The second main result is an application to Monge--Amp\`{e}re equations, via
the partial Legendre transform. Let $\kappa=f_{3,\sigma}^{2}$, i.e.
$\kappa(x)=|x|^{2(\ln^{(3)}\frac{1}{|x|})^{\sigma}}$ for $0<|x|\leq R$ and
$\kappa(0)=0$.

\begin{theorem}
[Monge--Amp\`{e}re]\label{th-MA}Let $0<\sigma<\frac{1}{2}$, $0<R<e^{-e}$ and
$B=B_{\mathbb{E}}(0,R)\subset\mathbb{R}^{2}$. Suppose $u\in C^{1,1}%
(\overline{B})$ is convex and%
\[
\det D^{2}u(x,y)=\kappa(x)\qquad\text{a.e. }(x,y)\in B,
\]
and that $y\mapsto u(0,y)$ is strictly convex on $(-R,R)$. Then $u\in
C^{2}(B)$.
\end{theorem}

Theorem \ref{th-geom-cont} is a consequence of an abstract continuity theorem
(Theorem \ref{th-abs-cont} in Section \ref{sec-main}). That theorem assumes:

\begin{itemize}
\item an Orlicz--Sobolev inequality with a superradius;

\item a $(1,1)$-Poincar\'{e} inequality;

\item a growth condition relating the superradius to how fast the volume of
metric balls changes with the radius;

\item a Dini condition on the right hand side.
\end{itemize}

\subsection{Method}

The local estimates come from Part I \cite{KoRiSaSh25}. There, a Moser
iteration adapted to the exp-log Young functions $\Phi_{m}$ gives $L^{2}%
$--$L^{\infty}$ bounds:

\begin{itemize}
\item for subsolutions;

\item for negative powers $(u+k)^{\beta}$, $\beta<0$, of nonnegative supersolutions.
\end{itemize}

The constants are $\exp(C(\ln\frac{\varphi(r)}{(1-\nu)r})^{m})$, where
$\varphi$ is the superradius (Definition \ref{def superradius}).

From the bounds for negative powers, a weak-$L^{1}$ estimate for $\log(u+k)$,
and a version of the Bombieri--Giusti lemma \cite{BoGi72}, we obtain a
\emph{lower half-Harnack inequality} with explicit dependence on the radius
(Theorem \ref{th-lowH}):%
\[
\operatorname*{ess\,inf}_{B(x,R_{2})}(u+k)\geq b(R)^{-1}\exp%
%TCIMACRO{\TeXButton{dashint}{\dashint}}%
%BeginExpansion
\dashint
%EndExpansion
_{B(x,R_{1})}\log(u+k).
\]
Here $R_{2}<R_{1}<R$ differ by the amount by which the radius has to shrink to
halve (up to a constant) the volume of the ball.

From this, a short argument in the spirit of De Giorgi \cite{DeG57} gives a
\emph{measure-to-pointwise} estimate (Lemma \ref{lem-meas}): if a nonnegative
supersolution is at least $1$ on half of $B(x,R_{1})$, then it is bounded
below by $\varepsilon_{0}(R)>0$ on $B(x,R_{2})$. This is the elliptic
counterpart of Landis' growth lemma \cite{Lan98} and of DiBenedetto's
expansion of positivity \cite{DiGiVe12}. Recent expositions with this
terminology are \cite{Kas07}, \cite{BrMo25}. Applied to $u-\inf u$ or to $\sup
u-u$, it gives the oscillation estimate (Theorem \ref{th-osc})%
\[
\operatorname*{ess}_{B(x,R_{2})}u\leq\left(  1-\tfrac{1}{2}\varepsilon
_{0}(R)\right)  \operatorname*{ess}_{B(x,R)}u+\mathcal{N}_{x}(R),
\]
where $\mathcal{N}_{x}(R)$ measures the right hand side (see Definition
\ref{def A admiss}).

The constant $\varepsilon_{0}(R)$ degenerates as $R\rightarrow0$. The growth
condition is exactly what guarantees $\sum_{j}\varepsilon_{0}(\tau
^{j}R)=\infty$. Together with the Dini condition $\sum_{j}\mathcal{N}_{x}%
(\tau^{j}R)<\infty$, this forces the oscillation to vanish yielding continuity.

We emphasize that only the \emph{lower} half of the Harnack inequality is
used. A full Harnack inequality would also require bounds for small positive
powers of solutions with constants of Bombieri--Giusti type. We do not know
whether such bounds hold in the present generality (see Remark \ref{rem-upper}).

\subsection{Relation to previous work}

In \cite{KoRiSaSh19} we obtained boundedness of weak solutions of
$-\operatorname{div}\mathcal{A}(x,u)\nabla u=\phi$, and continuity when
$\phi\equiv0$, using a De Giorgi iteration with the logarithmic Young
functions $t(\ln t)^{N}$. In Part I \cite{KoRiSaSh25} we implemented a Moser
iteration with the exp-log functions $\Phi_{m}$, under correspondingly
stronger hypotheses. The present paper treats inhomogeneous right hand sides
with a drift term and proves continuity. This is the first continuity result
for degenerate quasilinear equations with nonzero right hand side and infinite degeneracy.

Few related results have been obtained by other authors. The two most recent
and relevant ones are \cite{CrRo21} and \cite{DiFaMoRo23}. In \cite{CrRo21}
the authors obtain boundedness of weak solutions to a class of degenerate
elliptic Dirichlet problems, adapting the De Giorgi technique of
\cite{KoRiSaSh19}.

\begin{itemize}
\item Their results are abstract, assuming a degenerate Sobolev inequality.
They are similar to, but incomparable with, our abstract results.

\item They allow a much larger class of inhomogeneous data.

\item They do not obtain Harnack inequalities or continuity.

\item There are no geometric theorems, which would require verifying
hypotheses such as a Sobolev inequality.
\end{itemize}

They also remark \cite[page 5]{CrRo21} that \textquotedblleft We were unable
to adapt Moser iteration to work in the context of Orlicz norms, and it
remains an open question whether such an approach is possible in this
setting.\textquotedblright\ The Moser iteration of Part I, which we use here,
addresses that question. In \cite{DiFaMoRo23} the authors use Moser iteration
to obtain abstract Harnack inequalities and H\"{o}lder continuity for
quasilinear degenerate equations.

\begin{itemize}
\item As in \cite{CrRo21}, the relevant weighted Sobolev and Poincar\'{e}
inequalities and the doubling property of the metric balls are assumed.

\item There are no geometric theorems, so their results are also incomparable
with ours.
\end{itemize}

Both \cite{CrRo21} and \cite{DiFaMoRo23} use $(q,p)$ Sobolev inequalities with
$q>p$. Here only an Orlicz gain is available. This is what allows
non-doubling, infinitely degenerate geometries (see also \cite{KoMaRi}), at
the price of radius-dependent constants.

\subsection{Organization}

\begin{itemize}
\item Section \ref{sec-setting}: the abstract setting.

\item Section \ref{sec-main}: the abstract results and a roadmap of their proofs.

\item Section \ref{sec-local}: local estimates, quoted from Part I or proved
here (an interpolation lemma and the weak $\log$ estimate).

\item Section \ref{sec-BG}: the Bombieri--Giusti lemma.

\item Section \ref{sec-lowH}: the lower half-Harnack inequality, the
measure-to-pointwise lemma and the oscillation estimate.

\item Section \ref{sec-cont}: proof of the abstract continuity theorem.

\item Section \ref{sec-geom}: the geometries $F_{k,\sigma}$ and the proof of
Theorem \ref{th-geom-cont}.

\item Section \ref{sec-MA}: proof of Theorem \ref{th-MA}.

\item Appendix \ref{sec-app}: properties of $\Phi_{m}$.
\end{itemize}

\section{The abstract setting\label{sec-setting}}

Throughout, $\Omega\subset\mathbb{R}^{n}$ is a bounded domain and $A$ is a
nonnegative symmetric matrix function on $\Omega$ with $\sqrt{A}$ bounded and
Lipschitz. We write $|E|$ for the Lebesgue measure of $E$ and $%
%TCIMACRO{\TeXButton{dashint}{\dashint}}%
%BeginExpansion
\dashint
%EndExpansion
_{E}g=|E|^{-1}\int_{E}g$.

\subsection{Degenerate Sobolev spaces and weak solutions}

The space $W_{A}^{1,2}(\Omega)$ consists of all $v\in L^{2}(\Omega)$ whose
distributional $A$-gradient belongs to $L^{2}(\Omega;\mathbb{R}^{n})$, with
norm%
\[
\Vert v\Vert_{W_{A}^{1,2}}=\left( \int_{\Omega}|v|^{2}+|\nabla_{A}%
v|^{2}\right) ^{1/2}.
\]
Here $\vec{F}=\nabla_{A}v$ in the weak sense means%
\[
\int_{\Omega}\vec{F}\cdot\vec{V}\,dx=-\int_{\Omega}v\,\operatorname{div}%
_{A}\vec{V}\,dx\qquad\text{for all }\vec{V}\in\operatorname{Lip}_{\mathrm{c}%
}(\Omega;\mathbb{R}^{n}).
\]
The right side makes sense because $\operatorname{div}_{A}\vec{V}\in
L^{\infty}$ when $\sqrt{A}$ is Lipschitz. The space $W_{A}^{1,2}(\Omega)$ is a
Hilbert space \cite[Theorem 2]{SaWh3}.

The Lipschitz continuity of $\sqrt{A}$ is what allows us to approximate
functions in $W_{A}^{1,2}$ by smooth functions. This rests on the classical
commutator lemma of Friedrichs \cite{Fri44}. We collect what we need in the
following lemma; these results are also available from \cite{SaWh3}, we
include the short proofs here for clarity and completeness. Throughout,
$\rho_{\varepsilon}(z)=\varepsilon^{-n}\rho(z/\varepsilon)$ is a standard
mollifier, with $\rho\in C_{c}^{\infty}(B_{\mathbb{E}}(0,1))$, $\rho\geq0$ and
$\int\rho=1$. We write $a_{ij}$ for the entries of $\sqrt{A}$.

\begin{lemma}
[Friedrichs approximation in $W_{A}^{1,2}$]\label{lem-Friedrichs}Assume
$\sqrt{A}$ is bounded and Lipschitz in $\Omega$.

\begin{enumerate}
\item[(a)] \emph{(Friedrichs' lemma.)} Let $b\in W^{1,\infty}(\mathbb{R}^{n})$
and $v\in L^{2}(\mathbb{R}^{n})$, and define the distribution $b\,\partial
_{j}v=\partial_{j}(bv)-(\partial_{j}b)v$. Then%
\[
C_{\varepsilon}v=b\,\partial_{j}(v\ast\rho_{\varepsilon})-(b\,\partial
_{j}v)\ast\rho_{\varepsilon}%
\]
belongs to $L^{2}$ and satisfies $\Vert C_{\varepsilon}v\Vert_{L^{2}}%
\leq(c_{\rho}+1)\Vert\nabla b\Vert_{\infty}\Vert v\Vert_{L^{2}}$, where
$c_{\rho}=\int|z||\nabla\rho(z)|\,dz$. Moreover $C_{\varepsilon}v\rightarrow0$
in $L^{2}$ as $\varepsilon\rightarrow0$.

\item[(b)] \emph{(Compact support.)} Let $v\in W_{A}^{1,2}(\Omega)$ have
compact support in an open set $U\Subset\Omega$. Then for small $\varepsilon$
the function $v_{\varepsilon}=v\ast\rho_{\varepsilon}$ lies in $C_{c}^{\infty
}(U)$, and $v_{\varepsilon}\rightarrow v$, $\nabla_{A}v_{\varepsilon
}\rightarrow\nabla_{A}v$ in $L^{2}$. If $v\geq0$, then $v_{\varepsilon}\geq0$.

\item[(c)] \emph{(Local approximation and product rule.)} Let $U\subset\Omega$
be open, $v\in W_{A}^{1,2}(U)$ and $\psi\in\operatorname{Lip}(U)$ bounded.
Then $\psi v\in W_{A}^{1,2}(U)$ and $\nabla_{A}(\psi v)=\psi\nabla_{A}%
v+v\sqrt{A}\nabla\psi$. Moreover, for every compact $K\subset U$ there are
$v_{\varepsilon}\in C_{c}^{\infty}(U)$ with $v_{\varepsilon}\rightarrow v$ and
$\nabla_{A}v_{\varepsilon}\rightarrow\nabla_{A}v$ in $L^{2}(K)$.

\item[(d)] \emph{(Chain rule.)} Let $U\subset\Omega$ be open, $v\in
W_{A}^{1,2}(U)$, and let $F\in C^{1}(\mathbb{R})$ with $F^{\prime}$ bounded.
Then $F(v)\in W_{A,\mathrm{loc}}^{1,2}(U)$ and $\nabla_{A}F(v)=F^{\prime
}(v)\nabla_{A}v$.
\end{enumerate}
\end{lemma}

\begin{proof}
(a) A direct computation gives%
\[
C_{\varepsilon}v(x)=\int\left(  b(x)-b(y)\right)  v(y)\,\partial_{j}%
\rho_{\varepsilon}(x-y)\,dy+\left(  (\partial_{j}b)v\right)  \ast
\rho_{\varepsilon}(x).
\]
In the first term $|b(x)-b(y)||\partial_{j}\rho_{\varepsilon}(x-y)|\leq
\Vert\nabla b\Vert_{\infty}k_{\varepsilon}(x-y)$, where $k_{\varepsilon
}(z)=|z||\nabla\rho_{\varepsilon}(z)|$ has $\Vert k_{\varepsilon}\Vert_{L^{1}%
}=c_{\rho}$. Young's inequality for convolutions gives the bound. If $v\in
C_{c}^{\infty}$, then $b\,\partial_{j}(v\ast\rho_{\varepsilon})\rightarrow
b\,\partial_{j}v$ and $(b\,\partial_{j}v)\ast\rho_{\varepsilon}\rightarrow
b\,\partial_{j}v$ in $L^{2}$, so $C_{\varepsilon}v\rightarrow0$. The general
case follows from the uniform bound and the density of $C_{c}^{\infty}$ in
$L^{2}$.

(b) Extend each $a_{ij}$ to a bounded Lipschitz function on $\mathbb{R}^{n}$
(McShane extension), and extend $v$ and $\nabla_{A}v$ by zero. Since $v$
vanishes near $\partial U$, the weak definition of $\nabla_{A}v$ holds in
$\mathbb{R}^{n}$. Testing it with $\vec{V}=\varphi e_{i}$, and using the
symmetry of $\sqrt{A}$, gives in the sense of distributions%
\[
(\nabla_{A}v)_{i}=\sum_{j}\left(  \partial_{j}(a_{ij}v)-(\partial_{j}%
a_{ij})v\right)  =\sum_{j}a_{ij}\,\partial_{j}v,
\]
with the convention of (a). Hence%
\[
(\nabla_{A}v_{\varepsilon})_{i}=\sum_{j}a_{ij}\partial_{j}(v\ast
\rho_{\varepsilon})=(\nabla_{A}v)_{i}\ast\rho_{\varepsilon}+\sum
_{j}C_{\varepsilon}^{ij}v,
\]
where $C_{\varepsilon}^{ij}$ is the commutator of (a) with $b=a_{ij}$. The
first term tends to $(\nabla_{A}v)_{i}$ in $L^{2}$ and the second to $0$. For
small $\varepsilon$, $\operatorname{supp}v_{\varepsilon}\subset
\operatorname{supp}v+B_{\mathbb{E}}(0,\varepsilon)\subset U$. Positivity is
preserved by convolution with $\rho_{\varepsilon}\geq0$.

(c) Product rule. First let $\psi\in C^{\infty}$ be bounded with bounded
gradient. For $\vec{V}\in\operatorname{Lip}_{\mathrm{c}}(U;\mathbb{R}^{n})$,
$\operatorname{div}_{A}(\psi\vec{V})=\psi\operatorname{div}_{A}\vec{V}%
+\sqrt{A}\vec{V}\cdot\nabla\psi$. Hence%
\[
-\int\psi v\,\operatorname{div}_{A}\vec{V}=-\int v\,\operatorname{div}%
_{A}(\psi\vec{V})+\int v\sqrt{A}\nabla\psi\cdot\vec{V}=\int\left( \psi
\nabla_{A}v+v\sqrt{A}\nabla\psi\right) \cdot\vec{V}.
\]
For Lipschitz $\psi$, apply this to the mollifications $\psi_{\delta}=\psi
\ast\rho_{\delta}$ on compact subsets. Then $\psi_{\delta}\rightarrow\psi$
uniformly and $\nabla\psi_{\delta}\rightarrow\nabla\psi$ a.e. and boundedly,
and dominated convergence gives the identity.

Local approximation. Given $K$, choose $\chi\in C_{c}^{\infty}(U)$ with
$\chi=1$ near $K$. By the product rule, $\chi v\in W_{A}^{1,2}$ has compact
support in $U$, so (b) applies to it. On $K$ we have $\chi v=v$ and
$\nabla_{A}(\chi v)=\nabla_{A}v$.

(d) Let $K\subset U$ be compact and $v_{\varepsilon}$ as in (c). Passing to a
subsequence, $v_{\varepsilon}\rightarrow v$ a.e. on $K$. Then
$F(v_{\varepsilon})\rightarrow F(v)$ in $L^{2}(K)$. Moreover%
\[
\nabla_{A}F(v_{\varepsilon})=F^{\prime}(v_{\varepsilon})\nabla_{A}%
v_{\varepsilon}=F^{\prime}(v_{\varepsilon})(\nabla_{A}v_{\varepsilon}%
-\nabla_{A}v)+F^{\prime}(v_{\varepsilon})\nabla_{A}v\rightarrow F^{\prime
}(v)\nabla_{A}v
\]
in $L^{2}(K)$: the first term because $F^{\prime}$ is bounded, the second by
dominated convergence because $F^{\prime}$ is continuous. Passing to the limit
in the weak definition of $\nabla_{A}F(v_{\varepsilon})$, for test fields
supported in the interior of $K$, proves (d).
\end{proof}

\begin{remark}
\label{rem-Friedrichs-cases}The Lipschitz condition on $\sqrt{A}$ is used only
in (a)--(b). For merely bounded measurable degenerate matrices, smooth
functions need not be dense in $W_{A}^{1,2}$. In (d) one may also take $F$
piecewise $C^{1}$, with finitely many corners and $F^{\prime}$ bounded, by
approximating $F$ uniformly by $C^{1}$ functions with uniformly bounded
derivatives. We only need $F(t)=1/(t+k)$ and $F(t)=\log(t+k)$ on $[0,\infty)$,
$k>0$, each extended to a $C^{1}$ function on $\mathbb{R}$ with bounded
derivative. Since these are applied to $u\geq0$, the extension is irrelevant.
\end{remark}

\begin{definition}
[Weak solutions]\label{def weak solution}Let $\phi_{0}\in L_{\mathrm{loc}}%
^{2}(\Omega)$, $\vec{\phi}_{1}\in L_{\mathrm{loc}}^{2}(\Omega;\mathbb{R}^{n})$
and $\tilde{A}\in\mathfrak{A}(A,\Lambda,\lambda)$. A function $u\in
W_{A}^{1,2}(\Omega)$ is a \emph{weak solution} of $L_{\tilde{A}}u=\phi
_{0}-\operatorname{div}_{A}\vec{\phi}_{1}$ in $\Omega$ if%
\begin{equation}
\int_{\Omega}\tilde{A}\nabla u\cdot\nabla w\,dx=\int_{\Omega}\phi
_{0}w\,dx+\int_{\Omega}\vec{\phi}_{1}\cdot\nabla_{A}w\,dx\qquad\text{for all
}w\in\operatorname{Lip}_{\mathrm{c}}(\Omega).\label{equation}%
\end{equation}
It is a \emph{weak subsolution} (resp. \emph{supersolution}) if ``$=$'' is
replaced by ``$\leq$'' (resp. ``$\geq$'') for all nonnegative $w\in
\operatorname{Lip}_{\mathrm{c}}(\Omega)$. A weak solution of the quasilinear
equation (\ref{eq-quasi}) is a weak solution of (\ref{eq-linear}) with
$\tilde{A}=\mathcal{A}(\cdot,u(\cdot))$.
\end{definition}

By (\ref{struc_0}), $|\sqrt{\tilde{A}}\xi|^{2}\leq\Lambda|\sqrt{A}\xi|^{2}$,
so the left side of (\ref{equation}) is absolutely convergent.

If $u$ is a solution of (\ref{eq-linear}), then $u-c$ is a solution of the
same equation for every constant $c$, and $-u$ is a solution of the equation
with right hand side $-(\phi_{0}-\operatorname{div}_{A}\vec{\phi}_{1})$. In
the quasilinear case we always first freeze $\tilde{A}=\mathcal{A}%
(\cdot,u(\cdot))$, so these operations are legitimate.

\begin{remark}
[drift through $\mathcal{A}$]\label{rem-drift-A}Part I \cite{KoRiSaSh25}
writes the drift as $\operatorname{div}_{\mathcal{A}}\vec{\phi}_{1}$, i.e. it
uses $\vec{\phi}_{1}\cdot\sqrt{\tilde{A}}\nabla w$ in (\ref{equation}). The
two conventions are equivalent up to constants. Indeed, by (\ref{struc_0}) we
have $\ker\sqrt{A}\subset\ker\sqrt{\tilde{A}}$, so $\sqrt{\tilde{A}}=M\sqrt
{A}$ for a measurable matrix $M $ with $|M|\leq\sqrt{\Lambda}$. Hence
$\vec{\phi}_{1}\cdot\sqrt{\tilde{A}}\nabla w=(M^{\mathrm{tr}}\vec{\phi}%
_{1})\cdot\nabla_{A}w$. All the norms of the drift below change at most by the
factor $\sqrt{\Lambda}$.
\end{remark}

\noindent\emph{Convention.} With $M$ as in Remark \ref{rem-drift-A}, we
interpret $\tilde{A}\nabla u\cdot\nabla w$ as $(M\nabla_{A}u)\cdot(M\nabla
_{A}w)$. This is well defined for $u,w\in W_{A}^{1,2}$, and agrees with the
usual expression for Lipschitz functions.

\begin{corollary}
[admissible test functions]\label{cor-test}Let $u$ be a weak solution (resp.
subsolution, supersolution) of $L_{\tilde{A}}u=\phi_{0}-\operatorname{div}%
_{A}\vec{\phi}_{1}$ in $\Omega$. Then (\ref{equation}) (resp. the inequality
``$\leq$'', ``$\geq$'') holds for every $w\in W_{A}^{1,2}(\Omega)$ with
compact support in $\Omega$ (resp. every such $w\geq0$).
\end{corollary}

\begin{proof}
Let $U\Subset\Omega$ contain $\operatorname{supp}w$, and let $w_{\varepsilon
}\in C_{c}^{\infty}(U)$ be as in Lemma \ref{lem-Friedrichs}(b). Then
$w_{\varepsilon}\geq0$ if $w\geq0$, and (\ref{equation}) (or the inequality)
holds for $w_{\varepsilon}$. As $\varepsilon\rightarrow0$:

\begin{itemize}
\item $M\nabla_{A}w_{\varepsilon}\rightarrow M\nabla_{A}w$ in $L^{2}$, since
$|M|\leq\sqrt{\Lambda}$, so the left side converges;

\item $\int\phi_{0}w_{\varepsilon}\rightarrow\int\phi_{0}w$, since $\phi
_{0}\in L^{2}(U)$ and $w_{\varepsilon}\rightarrow w$ in $L^{2}$;

\item $\int\vec{\phi}_{1}\cdot\nabla_{A}w_{\varepsilon}\rightarrow\int%
\vec{\phi}_{1}\cdot\nabla_{A}w$, since $\vec{\phi}_{1}\in L^{2}(U)$.
\end{itemize}
\end{proof}

\subsection{Metric balls, deductions and cutoff functions}

We fix a metric $d$ on $\Omega$ that is continuous with respect to the
Euclidean topology, with balls $B(x,\rho)=\{y\in\Omega:d(x,y)<\rho\}$. For a
fixed center we abbreviate $B_{\rho}=B(x,\rho)$. In the applications $d$ is
the Carnot--Carath\'{e}odory metric of $A$.

The volume of $B(x,\rho)$ may decrease very fast as $\rho$ decreases. The
following notion quantifies how much the radius must be shrunk to divide the
volume by a fixed factor.

\begin{definition}
[deduction function]\label{def deduction}Let $x\in\Omega$, $0<R_{0}\leq1$ and
$D\geq2$. A function $\delta=\delta_{x}:(0,R_{0}]\rightarrow(0,\infty)$ is a
$D$-\emph{deduction} at $x$ if, for $0<\rho\leq R_{0}$:

\begin{enumerate}
\item[(D1)] $\delta(\rho)\leq\rho/4$ and $|B(x,\rho)|\leq D\,|B(x,\rho
-\delta(\rho))|$;

\item[(D2)] $\delta(\rho^{\prime})\geq c_{\delta}\,\delta(\rho)$ for
$\rho/2\leq\rho^{\prime}\leq\rho$, with a constant $c_{\delta}\in(0,1] $;

\item[(D3)] $\rho\mapsto\delta(\rho)/\rho$ is nondecreasing.
\end{enumerate}
\end{definition}

The \emph{halving deduction} $\delta(\rho)$, defined by $|B(x,\rho-\delta
(\rho))|=\frac{1}{2}|B(x,\rho)|$, satisfies (D1) with $D=2$ whenever
$\delta(\rho)\leq\rho/4$. In the geometries of Section \ref{sec-geom} it is
more convenient to use the explicit function $\delta(\rho)=c_{0}/|F^{\prime
}(\rho)|$, which satisfies (D1)--(D3) for a suitable $D$.

Given a $D$-deduction $\delta$ and $0<R\leq R_{0}$ we define the radii%
\begin{equation}
R_{1}=R-\delta(R),\qquad R_{2}=R_{1}-\delta(R_{1}).\label{radii}%
\end{equation}
By (D1), $R_{2}\geq\frac{9}{16}R$ and%
\begin{equation}
|B_{R}|\leq D|B_{R_{1}}|\leq D^{2}|B_{R_{2}}|.\label{radii volume}%
\end{equation}
By (D3), $\delta(\rho)\leq\delta(R)$ for $\rho\leq R$, hence%
\begin{equation}
R_{2}\geq R-2\delta(R)=\left( 1-2\frac{\delta(R)}{R}\right) R.\label{R2 lower}%
\end{equation}

\begin{definition}
[cutoff functions]\label{def cutoff}We say that \emph{cutoff functions exist
at} $(x,R_{0})$, with constant $C_{c}$, if for all $0<\rho<\rho^{\prime}\leq
R_{0}$ there is $\psi\in\operatorname{Lip}_{\mathrm{c}}(B(x,\rho^{\prime}))$
with%
\[
0\leq\psi\leq1,\qquad\psi=1\text{ on }B(x,\rho),\qquad\Vert\nabla_{A}\psi
\Vert_{L^{\infty}}\leq\frac{C_{c}}{\rho^{\prime}-\rho}.
\]

\end{definition}

Taking $\rho_{j}\searrow\nu\rho$ with $\rho_{j}-\rho_{j+1}\approx(1-\nu
)\rho\,j^{-1-\varepsilon}$, this gives the standard sequences of accumulating
Lipschitz cutoff functions used in Part I \cite[Definition 2.3]{KoRiSaSh25}.
The following lemma, essentially from \cite{SaWh4}, provides cutoff functions
for Carnot--Carath\'{e}odory metrics.

\begin{lemma}
\label{spec_cutoff_lemma}Let $A$ be continuous, and let $d$ be the
Carnot--Carath\'{e}odory (subunit) metric of $A$. Assume that $d$ is
topologically equivalent to the Euclidean metric. Then for every $x$ with
$\overline{B(x,R_{0})}\subset\Omega$, cutoff functions exist at $(x,R_{0})$
with $C_{c}=\frac{4}{3}\sqrt{n}$.
\end{lemma}

\begin{proof}
For $\varepsilon>0$ let $d^{\varepsilon}$ be the subunit metric of
$A+\varepsilon^{2}I$. By \cite[Lemma 65]{SaWh4}, $|\nabla_{A}d^{\varepsilon
}(\cdot,x)|\leq\sqrt{n}$ uniformly in $\varepsilon$. Moreover $d^{\varepsilon
}(\cdot,x)\nearrow d(\cdot,x)$ as $\varepsilon\searrow0$. Since $d$ is
continuous, the convergence is uniform on compact sets (Dini's theorem).

Let $g(t)=1$ for $t\leq\rho$, $g(t)=0$ for $t\geq\rho^{\prime}-\frac
{\rho^{\prime}-\rho}{4}$, and $g$ linear in between. Put $\psi
=g(d^{\varepsilon}(\cdot,x))$. Since $d^{\varepsilon}\leq d$, $\psi=1$ on
$B(x,\rho)$. For $\varepsilon$ small, uniform convergence on $\overline
{B(x,R_{0})}$ gives $d^{\varepsilon}\geq d-\frac{\rho^{\prime}-\rho}{8}$
there, so $\operatorname{supp}\psi\subset\{d\leq\rho^{\prime}-\frac
{\rho^{\prime}-\rho}{8}\}\subset B(x,\rho^{\prime})$. Finally $|\nabla_{A}%
\psi|\leq\Vert g^{\prime}\Vert_{\infty}\sqrt{n}=\frac{4\sqrt{n}}%
{3(\rho^{\prime}-\rho)}$.
\end{proof}

\begin{remark}
Continuity of $A$ cannot easily be dropped. Zhong \cite{Zhong} constructs a
discontinuous solution of a degenerate linear elliptic equation with a
discontinuous coefficient matrix.
\end{remark}

\subsection{Young functions and Orlicz--Sobolev
inequalities\label{subsection OS}}

A \emph{Young function} is a convex increasing $\Phi:[0,\infty)\rightarrow
\lbrack0,\infty)$ with $\Phi(0)=0$. On a probability space $(X,\mu)$ the
Luxemburg norm is%
\[
\Vert f\Vert_{L^{\Phi}(\mu)}=\inf\left\{ t>0:\int_{X}\Phi\left( \frac{|f|}%
{t}\right) d\mu\leq1\right\} .
\]
The conjugate function is $\Phi^{\ast}(s)=\sup_{t\geq0}(st-\Phi(t))$, and
H\"{o}lder's inequality reads%
\begin{equation}
\int|fg|\,d\mu\leq2\Vert f\Vert_{L^{\Phi}(\mu)}\Vert g\Vert_{L^{\Phi^{\ast}%
}(\mu)}.\label{Holder Orlicz}%
\end{equation}
The \emph{nonhomogeneous quasi-norm} is%
\[
\Vert f\Vert_{\mathcal{D}^{\Phi}(\mu)}=\Phi^{-1}\left( \int_{X}\Phi
(|f|)\,d\mu\right) .
\]
We say that $\Phi$ is \emph{submultiplicative} if $\Phi(st)\leq\Phi(s)\Phi
(t)$, and \emph{supermultiplicative} if $\Phi(s)\Phi(t)\leq\Phi(st)$, for all
$s,t\geq0$.

For $m>1$ we use the exp-log Young functions of Part I \cite{KoRiSaSh25}:%
\begin{equation}
\Phi_{m}(t)=\left\{
\begin{array}
[c]{ll}%
\exp\left(  ((\ln t)^{\frac{1}{m}}+1)^{m}\right)  & \text{if }t\geq e^{2^{m}%
},\\
e^{3^{m}-2^{m}}\,t & \text{if }0\leq t\leq e^{2^{m}}.
\end{array}
\right. \label{def Phi m}%
\end{equation}
The two expressions agree at $t=e^{2^{m}}$. The slope of the exponential
branch there is $(3/2)^{m-1}e^{3^{m}-2^{m}}$, which exceeds that of the linear
branch, so $\Phi_{m}$ is convex. For large $t$,%
\[
\Phi_{m}(t)=t\exp\left(  m(\ln t)^{\frac{m-1}{m}}+O((\ln t)^{\frac{m-2}{m}%
})\right)  ,
\]
so $\Phi_{m}$ grows faster than $t(\ln t)^{N}$ for every $N$, but slower than
$t^{1+\varepsilon}$ for every $\varepsilon>0$. These functions go back to
Cruz-Uribe and Fiorenza \cite{CUF02} (see also \cite{CMP11}). By \cite[Section
3]{KoRiSaSh25}, $\Phi_{m}$ is submultiplicative.

\begin{lemma}
\label{lem super conj}If $\Phi$ is a submultiplicative Young function, then
$\Phi^{\ast}$ is supermultiplicative.
\end{lemma}

\begin{proof}
Let $s,t>0$. If $\Phi^{\ast}(s)=0$ or $\Phi^{\ast}(t)=0$ there is nothing to
prove. Otherwise, given $0<\varepsilon<\min\{\Phi^{\ast}(s),\Phi^{\ast}(t)\}$,
choose $a,b\geq0$ with $sa-\Phi(a)\geq\Phi^{\ast}(s)-\varepsilon>0$ and
$tb-\Phi(b)\geq\Phi^{\ast}(t)-\varepsilon>0$. In particular $sa\geq\Phi(a)$
and $tb\geq\Phi(b)$. Then%
\begin{align*}
\Phi^{\ast}(st)  & \geq st\,ab-\Phi(ab)\geq(sa)(tb)-\Phi(a)\Phi(b)\\
& =(sa-\Phi(a))(tb-\Phi(b))+sa\,\Phi(b)+tb\,\Phi(a)-2\Phi(a)\Phi(b)\\
& \geq(\Phi^{\ast}(s)-\varepsilon)(\Phi^{\ast}(t)-\varepsilon).
\end{align*}
Letting $\varepsilon\rightarrow0$ gives $\Phi^{\ast}(st)\geq\Phi^{\ast}%
(s)\Phi^{\ast}(t)$.
\end{proof}

\begin{definition}
[superradius]\label{def superradius}A \emph{superradius} on $(0,R_{0}]$ is a
nondecreasing function $\varphi:(0,R_{0}]\rightarrow(0,\infty)$ with
$\varphi(\rho)\geq\rho$ that is locally comparable to its radius: there is
$C_{\varphi}\geq1$ with%
\begin{equation}
\frac{\varphi(\rho)}{\rho}\leq C_{\varphi}\frac{\varphi(r)}{r}\qquad\text{for
}\frac{r}{2}\leq\rho\leq r\leq R_{0}.\label{superradius comp}%
\end{equation}

\end{definition}

\begin{remark}
In the infinitely degenerate examples $\varphi(\rho)/\rho\rightarrow\infty$ as
$\rho\rightarrow0$. So $\varphi(\rho)/\rho$ is \emph{decreasing}, and only the
local comparability (\ref{superradius comp}) can be expected.
\end{remark}

\begin{definition}
[Orlicz--Sobolev inequalities]\label{def OS}Let $x\in\Omega$, $\overline
{B(x,R_{0})}\subset\Omega$, let $\Phi$ be a Young function and $\varphi$ a
superradius on $(0,R_{0}]$.

\begin{enumerate}
\item The \emph{nonhomogeneous} $(\Phi,\varphi)$-\emph{Orlicz--Sobolev
inequality holds at} $(x,R_{0})$ if for all $0<\rho\leq R_{0}$ and
$w\in\operatorname{Lip}_{\mathrm{c}}(B(x,\rho))$,%
\begin{equation}
\Phi^{-1}\left(
%TCIMACRO{\TeXButton{dashint}{\dashint}}%
%BeginExpansion
\dashint
%EndExpansion
_{B(x,\rho)}\Phi(|w|)\right) \leq\varphi(\rho)%
%TCIMACRO{\TeXButton{dashint}{\dashint}}%
%BeginExpansion
\dashint
%EndExpansion
_{B(x,\rho)}|\nabla_{A}w|.\label{OS-NH ineq}%
\end{equation}

\item The (homogeneous) $(\Phi,\varphi)$-\emph{Orlicz--Sobolev inequality
holds at} $(x,R_{0})$ if for the same $\rho,w$,%
\begin{equation}
\Vert w\Vert_{L^{\Phi}(\mu_{x,\rho})}\leq\varphi(\rho)%
%TCIMACRO{\TeXButton{dashint}{\dashint}}%
%BeginExpansion
\dashint
%EndExpansion
_{B(x,\rho)}|\nabla_{A}w|,\qquad d\mu_{x,\rho}=\frac{\mathbf{1}_{B(x,\rho)}%
dy}{|B(x,\rho)|}.\label{OS ineq}%
\end{equation}

\end{enumerate}
\end{definition}

\begin{lemma}
\label{lem-NH-Lux}Let $\Phi$ be a Young function.

\begin{enumerate}
\item If (\ref{OS-NH ineq}) holds at $(x,R_{0})$, then (\ref{OS ineq}) holds
at $(x,R_{0})$ with $\varphi$ replaced by $\varphi/\Phi^{-1}(1)$.

\item If $\Phi$ is submultiplicative and (\ref{OS ineq}) holds at $(x,R_{0})
$, then (\ref{OS-NH ineq}) holds at $(x,R_{0})$.
\end{enumerate}
\end{lemma}

\begin{proof}
(1) Let $w\in\operatorname{Lip}_{\mathrm{c}}(B(x,\rho))$, $w\neq0$, and
$\lambda=\varphi(\rho)%
%TCIMACRO{\TeXButton{dashint}{\dashint}}%
%BeginExpansion
\dashint
%EndExpansion
|\nabla_{A}w|/\Phi^{-1}(1)$. Apply (\ref{OS-NH ineq}) to $w/\lambda$:
$\Phi^{-1}(%
%TCIMACRO{\TeXButton{dashint}{\dashint}}%
%BeginExpansion
\dashint
%EndExpansion
\Phi(|w|/\lambda))\leq\Phi^{-1}(1)$, so $%
%TCIMACRO{\TeXButton{dashint}{\dashint}}%
%BeginExpansion
\dashint
%EndExpansion
\Phi(|w|/\lambda)\leq1$ and $\Vert w\Vert_{L^{\Phi}(\mu_{x,\rho})}\leq\lambda$.

(2) Let $t=\Vert w\Vert_{L^{\Phi}(\mu_{x,\rho})}$. Submultiplicativity gives $%
%TCIMACRO{\TeXButton{dashint}{\dashint}}%
%BeginExpansion
\dashint
%EndExpansion
\Phi(|w|)\leq\Phi(t)%
%TCIMACRO{\TeXButton{dashint}{\dashint}}%
%BeginExpansion
\dashint
%EndExpansion
\Phi(|w|/t)\leq\Phi(t)$, so $\Phi^{-1}(%
%TCIMACRO{\TeXButton{dashint}{\dashint}}%
%BeginExpansion
\dashint
%EndExpansion
\Phi(|w|))\leq t$.
\end{proof}

\begin{remark}
[geometric examples]\label{remark-OS-geom}Let $A=\operatorname{diag}%
(1,f(x_{1})^{2})$ in the plane with $f=e^{-F}$. By \cite[Corollary
5.7]{KoRiSaSh25}, for $F=F_{k,\sigma}$ with $k\geq2$, $\sigma>0$ and every
$m>1$, the nonhomogeneous $(\Phi_{m},\varphi)$-Orlicz--Sobolev inequality
holds at the origin, with%
\begin{equation}
\varphi(r)=r\exp\left( C_{m}\left( \ln^{(k)}\frac{1}{r}\right) ^{\sigma
(m-1)}\right) ,\qquad0<r\leq r_{m,\sigma}.\label{phi-grow}%
\end{equation}
For $k=1$ the same holds when $0<\sigma<\frac{1}{m-1}$, with $\varphi
(r)=r\exp(C_{m}(\ln\frac{1}{r})^{\sigma(m-1)})$, and fails when $\sigma
>\frac{1}{m-1}$. In \cite[Proposition 80]{KoRiSaSh19} the corresponding
inequalities for the logarithmic Young functions $t(\ln t)^{N}$ were proved.
\end{remark}

\subsection{Poincar\'{e} inequality}

\begin{definition}
\label{def Poincare}The $(1,1)$-\emph{Poincar\'{e} inequality holds at}
$(x,R_{0})$ if there is $C_{P}$ such that for $0<\rho\leq R_{0}$ and
$g\in\operatorname{Lip}(\overline{B(x,\rho)})$,%
\begin{equation}%
%TCIMACRO{\TeXButton{dashint}{\dashint}}%
%BeginExpansion
\dashint
%EndExpansion
_{B(x,\rho)}\left| g-%
%TCIMACRO{\TeXButton{dashint}{\dashint}}%
%BeginExpansion
\dashint
%EndExpansion
_{B(x,\rho)}g\right| \leq C_{P}\,\rho%
%TCIMACRO{\TeXButton{dashint}{\dashint}}%
%BeginExpansion
\dashint
%EndExpansion
_{B(x,\rho)}|\nabla_{A}g|.\label{Poinc}%
\end{equation}

\end{definition}

The inequality (\ref{Poinc}) extends to $g\in W_{A}^{1,2}(U)$ for any open
$U\supset\overline{B(x,\rho)}$. Indeed, $\overline{B(x,\rho)}$ is a compact
subset of $U$ because $d$ is continuous. Lemma \ref{lem-Friedrichs}(c) with
$K=\overline{B(x,\rho)}$ gives smooth $g_{\varepsilon}$, hence Lipschitz on
$\overline{B(x,\rho)}$, with $g_{\varepsilon}\rightarrow g$ and $\nabla
_{A}g_{\varepsilon}\rightarrow\nabla_{A}g$ in $L^{2}(B(x,\rho))$. Both sides
of (\ref{Poinc}) then converge, since $L^{2}(B(x,\rho))\subset L^{1}%
(B(x,\rho))$.

\subsection{Right hand sides}

\begin{definition}
[admissible norm]\label{def A admiss}Let $B=B(x,\rho)\subset\Omega$. For
$\phi_{0}\in L^{2}(B)$ and $\vec{\phi}_{1}\in L^{2}(B;\mathbb{R}^{n})$ let
$\mathcal{N}^{0}(B)$ and $\mathcal{N}^{1}(B)$ be the smallest constants such
that, for all $w\in\operatorname{Lip}_{\mathrm{c}}(B)$,%
\[
\left| \int_{B}\phi_{0}w\right| \leq\mathcal{N}^{0}(B)\int_{B}|\nabla
_{A}w|,\qquad\left| \int_{B}\vec{\phi}_{1}\cdot\nabla_{A}w\right|
\leq\mathcal{N}^{1}(B)\int_{B}|\nabla_{A}w|.
\]
We write $\mathcal{N}_{x}(\rho)=\mathcal{N}^{0}(B(x,\rho))+\mathcal{N}%
^{1}(B(x,\rho))\in\lbrack0,\infty]$. The pair $(\phi_{0},\vec{\phi}_{1}) $ is
\emph{admissible} in $B$ if $\mathcal{N}_{x}(\rho)<\infty$.
\end{definition}

Clearly $\rho\mapsto\mathcal{N}_{x}(\rho)$ is nondecreasing. If $\mathcal{N}%
_{x}(\rho)<\infty$, then%
\begin{equation}
\left\vert \int_{B}\phi_{0}w+\int_{B}\vec{\phi}_{1}\cdot\nabla_{A}w\right\vert
\leq\mathcal{N}_{x}(\rho)\int_{B}|\nabla_{A}w|\label{admiss}%
\end{equation}
for every $w\in W_{A}^{1,2}(\Omega)$ with compact support in $B$. Indeed,
approximate $w$ by $w_{\varepsilon}\in C_{c}^{\infty}(B)$ as in Lemma
\ref{lem-Friedrichs}(b). Then $\int\phi_{0}w_{\varepsilon}$ and $\int\vec
{\phi}_{1}\cdot\nabla_{A}w_{\varepsilon}$ converge as in Corollary
\ref{cor-test}, and $\int_{B}|\nabla_{A}w_{\varepsilon}|\rightarrow\int%
_{B}|\nabla_{A}w|$ because $L^{2}(B)\subset L^{1}(B)$.

\begin{definition}
[Dini admissibility]\label{def Dini A adm}The pair $(\phi_{0},\vec{\phi}_{1})$
is \emph{Dini admissible at} $(x,R_{0})$ if $\mathcal{N}_{x}(R_{0})<\infty$
and $\sum_{j\geq0}\mathcal{N}_{x}(\tau^{j}R_{0})<\infty$ for some $\tau
\in(0,1)$.
\end{definition}

Since $\mathcal{N}_{x}$ is nondecreasing, the choice of $\tau$ is irrelevant.
Given $\tau,\tau^{\prime}\in(0,1)$, each interval $(\tau^{j+1}R_{0},\tau
^{j}R_{0}]$ contains at most $1+\ln\tau/\ln\tau^{\prime}$ points of the form
$\tau^{\prime\,i}R_{0}$, and $\mathcal{N}_{x}(\tau^{\prime\,i}R_{0}%
)\leq\mathcal{N}_{x}(\tau^{j}R_{0})$ for each of them.

\begin{lemma}
\label{lem:aux}Let $\Psi$ be a supermultiplicative Young function, $x\in
\Omega$ and $0<r\leq R$ with $B(x,R)\subset\Omega$. Then%
\[
\Vert f\Vert_{L^{\Psi}(\mu_{x,r})}\leq\Psi^{-1}\left( \frac{|B(x,R)|}%
{|B(x,r)|}\right) \Vert f\Vert_{L^{\Psi}(\mu_{x,R})}.
\]

\end{lemma}

\begin{proof}
Let $b=\Vert f\Vert_{L^{\Psi}(\mu_{x,R})}$, $\Gamma=|B(x,R)|/|B(x,r)|$ and
$d=b\Psi^{-1}(\Gamma)$. By supermultiplicativity,%
\[
\Gamma\,\Psi\left( \frac{|f|}{d}\right) =\Psi\left( \frac{|f|}{d}\right)
\Psi\left( \Psi^{-1}(\Gamma)\right) \leq\Psi\left( \frac{|f|}{b}\right) .
\]
Hence%
\[%
%TCIMACRO{\TeXButton{dashint}{\dashint}}%
%BeginExpansion
\dashint
%EndExpansion
_{B(x,r)}\Psi\left( \frac{|f|}{d}\right) \leq\frac{1}{\Gamma|B(x,r)|}%
\int_{B(x,r)}\Psi\left( \frac{|f|}{b}\right) \leq%
%TCIMACRO{\TeXButton{dashint}{\dashint}}%
%BeginExpansion
\dashint
%EndExpansion
_{B(x,R)}\Psi\left( \frac{|f|}{b}\right) \leq1,
\]
that is, $\Vert f\Vert_{L^{\Psi}(\mu_{x,r})}\leq d$.
\end{proof}

\begin{proposition}
[sufficient condition for $\phi_{0}$]\label{Dini-suff}Suppose the
nonhomogeneous $(\Phi,\varphi)$-Orlicz--Sobolev inequality holds at
$(x,R_{0})$ for a submultiplicative $\Phi$, and $\phi_{0}\in L^{\Phi^{\ast}%
}(B(x,R_{0}))$. Then for $0<\rho\leq R_{0}$,%
\[
\mathcal{N}^{0}(B(x,\rho))\leq\frac{2\varphi(\rho)}{\Phi^{-1}(1)}(\Phi^{\ast
})^{-1}\left( \frac{|B(x,R_{0})|}{|B(x,\rho)|}\right) \Vert\phi_{0}%
\Vert_{L^{\Phi^{\ast}}(\mu_{x,R_{0}})}.
\]
In particular $(\phi_{0},0)$ is Dini admissible at $(x,R_{0})$ provided the
\emph{superradius summability condition}%
\begin{equation}
\sum_{j\geq0}\varphi(\tau^{j}R_{0})\,(\Phi^{\ast})^{-1}\left( \frac
{|B(x,R_{0})|}{|B(x,\tau^{j}R_{0})|}\right) <\infty\label{super-summ}%
\end{equation}
holds for some $\tau\in(0,1)$.
\end{proposition}

\begin{proof}
Let $w\in\operatorname{Lip}_{\mathrm{c}}(B(x,\rho))$. By (\ref{Holder Orlicz})
and Lemma \ref{lem-NH-Lux}(1),%
\[
\left| \int_{B(x,\rho)}\phi_{0}w\right| \leq2|B(x,\rho)|\,\Vert w\Vert
_{L^{\Phi}(\mu_{x,\rho})}\Vert\phi_{0}\Vert_{L^{\Phi^{\ast}}(\mu_{x,\rho}%
)}\leq\frac{2\varphi(\rho)}{\Phi^{-1}(1)}\Vert\phi_{0}\Vert_{L^{\Phi^{\ast}%
}(\mu_{x,\rho})}\int|\nabla_{A}w|.
\]
By Lemma \ref{lem super conj}, $\Phi^{\ast}$ is supermultiplicative, and Lemma
\ref{lem:aux} with $\Psi=\Phi^{\ast}$ finishes the proof.
\end{proof}

For the drift the trivial bound is the relevant one:%
\begin{equation}
\mathcal{N}^{1}(B(x,\rho))\leq\Vert\vec{\phi}_{1}\Vert_{L^{\infty}(B(x,\rho
))}.\label{drift bound}%
\end{equation}
Hence $(0,\vec{\phi}_{1})$ is Dini admissible at $(x,R_{0})$ whenever
$\sum_{j}\Vert\vec{\phi}_{1}\Vert_{L^{\infty}(B(x,\tau^{j}R_{0}))}<\infty$.
Boundedness of $\vec{\phi}_{1}$ alone does not suffice.

\begin{remark}
\label{rem-drift-counter}Let $A=\operatorname{diag}(1,f(x_{1})^{2})$ as in
Theorem \ref{th-geom-cont}, $P$ a point on the line $\{x_{1}=0\}$, and
$\vec{\phi}_{1}=(\mathbf{1}_{\{x_{1}>0\}},0)$. Then by Green's theorem
$\int\vec{\phi}_{1}\cdot\nabla_{A}w=-\int w(0,x_{2})\,dx_{2}$ for any
$w\in{{\operatorname*{Lip}_{\mathrm{c}}}}\left(  B\left(  P,\varepsilon
\right)  \right)  $, $0<\varepsilon\leq\rho$. Take $w=h(x_{1})g(x_{2})$, where
$g\geq0$ is a fixed unit bump adapted to $\left(  -\varepsilon,\varepsilon
\right)  $ where $\varepsilon$ is such that $\left\{  P_{1}\right\}
\times\left[  -\varepsilon,\varepsilon\right]  \subset B\left(  P,\rho\right)
$, and $h$ is a tent of height $1$ and width $\varepsilon$. Then $\left\vert
\int\vec{\phi}_{1}\cdot\nabla_{A}w\right\vert =\left\vert \int w(0,x_{2}%
)\,dx_{2}\right\vert =\left\vert \int g\,dx_{2}\right\vert =1$, while
$\int_{B\left(  P,\varepsilon\right)  }|\partial_{1}w|~dx=2\int g~dx_{2}=2$,
and $\int f~\left\vert \partial_{2}w\right\vert ~dx\leq\varepsilon\sup
f\int\left\vert g^{\prime}\right\vert ~dx_{2}\lesssim\sup f_{\left\vert
x\right\vert <\varepsilon}\rightarrow0$ as $\varepsilon\rightarrow0$ because
of the rapid decay of $f$. Hence
\[
\mathcal{N}^{1}\left(  B\left(  P,\rho\right)  \right)  =\sup_{w\in
{{\operatorname*{Lip}_{\mathrm{c}}}}\left(  B\left(  P,\rho\right)  \right)
}\frac{\left\vert \int_{B\left(  P,\rho\right)  }\vec{\phi}_{1}\cdot\nabla
_{A}w\right\vert }{\int_{B\left(  P,\rho\right)  }\left\vert \nabla
_{A}w\right\vert }\geq\frac{1}{2}\text{ for every }\rho>0,
\]
and $\vec{\phi}_{1}$ is not Dini admissible at $P$. For uniformly elliptic
equations the drift contributes at scale $\rho$ an amount of order $\rho
\Vert\vec{\phi}_{1}\Vert_{\infty}$ in a scaling sense (see \cite[$k\left(
R\right)  $ in Theorems 8.17, 8.18, 8.22]{GiTr}), however, the quantity
$\mathcal{N}^{1}$ in Definition \ref{def A admiss} does not \emph{register}
the length factor $\rho$.
\end{remark}

\subsection{The growth condition}

Let $\delta$ be a $D$-deduction (Definition \ref{def deduction}) at $x$,
$\varphi$ a superradius (Definition \ref{def superradius}), we set%
\begin{equation}
\Theta(\rho)=\ln\frac{\varphi(\rho)}{\delta(\rho)}\geq\ln4,\qquad0<\rho\leq
R_{0}.\label{def Theta}%
\end{equation}
The lower bound in (\ref{def Theta}) follows from the definitions since
$\delta(\rho)\leq\rho/4$ and $\varphi(\rho)\geq\rho$.

\begin{definition}
[growth condition]\label{growth_cond}Given a $D$-deduction $\delta$, a
superradius ${\varphi} $ and $m>1$, we say that $(\delta,\varphi,m)$ satisfies
the \emph{growth condition} at $x$ if, for every $C\geq1$, every $\tau
\in(0,1)$ and every $0<R\leq R_{0}$,%
\begin{equation}
\sum_{j\geq0}\exp\left(  -C\exp\left(  C\,\Theta(\tau^{j}R)^{m}\right)
\right)  =\infty.\label{delta_growth}%
\end{equation}

\end{definition}

\begin{lemma}
\label{lem-growth-suff}Suppose that for every $\varepsilon>0$ there is
$\rho_{\varepsilon}>0$ such that%
\begin{equation}
\Theta(\rho)^{m}\leq\varepsilon\ln^{(3)}\frac{1}{\rho},\qquad0<\rho\leq
\rho_{\varepsilon}.\label{growth ln3}%
\end{equation}
Then the growth condition (\ref{delta_growth}) holds.
\end{lemma}

\begin{proof}
Fix $C\geq1$, $\tau$ and $R$, and let $\varepsilon=1/(2C)$. Put $\rho_{j}%
=\tau^{j}R$. For $j$ large we have $\rho_{j}\leq\min\{\rho_{\varepsilon
},e^{-e}\}$, and then%
\[
\exp(C\Theta(\rho_{j})^{m})\leq\exp\left(  \frac{1}{2}\ln^{(3)}\frac{1}%
{\rho_{j}}\right)  =\left(  \ln^{(2)}\frac{1}{\rho_{j}}\right)  ^{1/2}.
\]
Let $L=\max\{\ln\frac{1}{\tau},\ln\frac{1}{R},e\}$. Then $\ln\frac{1}{\rho
_{j}}\leq(j+1)L$, so%
\[
\left(  \ln^{(2)}\frac{1}{\rho_{j}}\right)  ^{1/2}\leq(\ln(j+1))^{1/2}+(\ln
L)^{1/2}.
\]
For $j$ large, $C(\ln(j+1))^{1/2}\leq\ln(j+1)$. Hence the $j$-th term of
(\ref{delta_growth}) is at least $e^{-C(\ln L)^{1/2}}(j+1)^{-1}$, and the
series diverges.
\end{proof}

\begin{remark}
[why three logarithms]\label{rem-three-logs}The third iterated logarithm in
(\ref{growth ln3}) is not an artifact of the proof. It comes from three
successive losses, and the growth condition (\ref{delta_growth}) is
essentially sharp for our method.

\begin{enumerate}
\item \emph{A double exponential in the positivity constant.} Part I gives
local bounds with constants $\exp(C\Theta^{m})$, of logarithmic type in
$\varphi/\delta$. The Bombieri--Giusti lemma turns the weak $L^{1}$
information on $\log u$ into a pointwise bound, and in doing so it
exponentiates the constant once more. So the lower half-Harnack constant
$b(R)$, and hence $\varepsilon_{0}(R)^{-1}$, is of size $\exp(C\exp
(C\Theta(R)^{m}))$. This is the first logarithm we must undo.

\item \emph{A divergent series.} Continuity requires $\prod_{j}(1-\frac{1}%
{2}\varepsilon_{0}(\rho_{j}))=0$, i.e. $\sum_{j}\varepsilon_{0}(\rho
_{j})=\infty$. Compared with the harmonic series, this essentially requires
$\varepsilon_{0}(\rho_{j})\gtrsim1/j$, i.e. $C\exp(C\Theta(\rho_{j}%
)^{m})\lesssim\ln j$. Removing the outer exponential costs a second logarithm:
$C\Theta(\rho_{j})^{m}\lesssim\ln\ln j$.

\item \emph{Geometric radii.} The radii decrease geometrically, $\rho_{j}%
=\tau^{j}R$, so $j\approx\ln\frac{1}{\rho_{j}}$. Expressed in the radius, the
condition becomes $C\Theta(\rho)^{m}\lesssim\ln^{(3)}\frac{1}{\rho}$.
\end{enumerate}

Since the structural constant $C$ is not explicit, we require $\Theta
(\rho)^{m}=o(\ln^{(3)}\frac{1}{\rho})$, which is (\ref{growth ln3}).
(Divergence would also follow from slightly slower decay such as
$\varepsilon_{0}(\rho_{j})\gtrsim1/(j\ln j)$, but this only affects
lower-order terms.)

For the geometries $F_{k,\sigma}$ of Section \ref{sec-geom}, with $L_{k}%
=\ln^{(k)}\frac{1}{\rho}$, we have $\Theta(\rho)\approx C_{m}L_{k}%
^{\sigma(m-1)}+\sigma L_{k+1}$ (see the proof of Lemma \ref{lem-geom-H}). Therefore:

\begin{itemize}
\item For $k\geq4$, every power of $L_{k}$ is $o(L_{3})$, and the condition
holds for every $\sigma>0$.

\item For $k=3$ we need $L_{3}^{\sigma m(m-1)}=o(L_{3})$, i.e. $\sigma
m(m-1)<1$. Part I requires $m>2$, so this gives exactly $0<\sigma<\frac{1}{2}
$ (letting $m\downarrow2$).

\item For $k=2$, $\Theta^{m}\approx L_{2}^{\sigma m(m-1)}$ is a power of
$\ln\ln\frac{1}{\rho}$. This is much larger than $\ln\ln\ln\frac{1}{\rho}$, so
the method fails for every $\sigma>0$.
\end{itemize}

Improving the range in $k$ or $\sigma$ would require removing one of the three
losses, for instance, obtaining a positivity constant bound with a single exponential.
\end{remark}

\section{The abstract results\label{sec-main}}

\subsection{Standing hypotheses}

Fix $x\in\Omega$ and $0<R_{0}\leq1$ with $\overline{B(x,R_{0})}$ compact in
$\Omega$. We use the following hypotheses at $(x,R_{0})$.

\begin{enumerate}
\item[(H1)] \emph{Structure.} $\sqrt{A}$ is bounded and Lipschitz,
$\mathcal{A}(\cdot,z)\in\mathfrak{A}(A,\Lambda,\lambda)$ for all $z$, where
$\mathfrak{A}(A,\Lambda,\lambda)$ is the class of nonnegative symmetric
matrices satisfying (\ref{struc_0}), and $d$ is a metric continuous with
respect to the Euclidean topology.

\item[(H2)] \emph{Cutoffs.} Cutoff functions exist at $(x,R_{0})$ (Definition
\ref{def cutoff}).

\item[(H3)] \emph{Poincar\'{e}.} The $(1,1)$-Poincar\'{e} inequality holds at
$(x,R_{0})$ (Definition \ref{def Poincare}).

\item[(H4)] \emph{Orlicz--Sobolev.} For some $m>2$ and some superradius
$\varphi$, the nonhomogeneous $(\Phi_{m},\varphi)$-Orlicz--Sobolev inequality
holds at $(x,R_{0})$ (Definition \ref{def OS}).

\item[(H5)] \emph{Deduction and growth.} There is a $D$-deduction $\delta$ at
$x$ on $(0,R_{0}]$ (Definition \ref{def deduction}) such that $(\delta
,\varphi,m)$ satisfies the growth condition (Definition \ref{growth_cond}).

\item[(H6)] \emph{Right hand side.} $(\phi_{0},\vec{\phi}_{1})$ is Dini
admissible at $(x,R_{0})$ (Definition \ref{def Dini A adm}).
\end{enumerate}

By Lemma \ref{lem-NH-Lux}, (H4) also gives the homogeneous inequality
(\ref{OS ineq}) for $\Phi_{m}$. Since $t(\ln t)^{N}\lesssim\Phi_{m}(t)$, it
also gives the logarithmic Orlicz--Sobolev inequalities of \cite{KoRiSaSh19},
with the same superradius.

Unless stated otherwise, \textquotedblleft$C$\textquotedblright\ denotes
constants depending only on $n,\lambda,\Lambda,m$, the constant of Part I in
Theorems \ref{th-PartI-sub} and \ref{L_infinity-beta} below, and $C_{c}%
,C_{P},C_{\varphi},D,c_{\delta}$. Such constants are called \emph{structural}.
Recall the radii $R_{1},R_{2}$ from (\ref{radii}) and $\Theta(R)=\ln
\frac{\varphi(R)}{\delta(R)}$ from (\ref{def Theta}), used to define the
growth condition (Definition \ref{growth_cond}).

\subsection{Statements}

\begin{theorem}
[lower half-Harnack inequality]\label{th-lowH}Assume (H1)--(H4), and let
$\delta$ be a $D$-deduction at $x$ on $(0,R_{0}]$. Let $0<R\leq R_{0}$, let
$u\geq0$ be a weak supersolution of (\ref{eq-linear}) in $B(x,R)$, and let
$k\geq\mathcal{N}_{x}(R)$ with $k>0$. Then%
\begin{equation}
\operatorname*{ess\,inf}_{B(x,R_{2})}(u+k)\geq\frac{1}{b(R)}\exp\left(
%TCIMACRO{\TeXButton{dashint}{\dashint}}%
%BeginExpansion
\dashint
%EndExpansion
_{B(x,R_{1})}\log(u+k)\right) ,\label{lowH}%
\end{equation}
where $b(R)\leq\exp\left[ C\exp(C\,\Theta(R)^{m})\right] $ with $C$ structural.
\end{theorem}

\begin{theorem}
[oscillation estimate]\label{th-osc}Assume (H1)--(H4), and let $\delta$ be a
$D$-deduction at $x$ on $(0,R_{0}]$. Let $0<R\leq R_{0}$ and let $u$ be a weak
solution of (\ref{eq-quasi}) in $B(x,R)$ that is essentially bounded in
$B(x,R)$. Then%
\begin{equation}
\operatorname*{ess\,osc}_{B(x,R_{2})}u\leq\left( 1-\frac{\varepsilon_{0}%
(R)}{2}\right) \operatorname*{ess\,osc}_{B(x,R)}u+\mathcal{N}_{x}%
(R),\label{osc}%
\end{equation}
where $\varepsilon_{0}(R)\in(0,1]$ satisfies $\varepsilon_{0}(R)\geq\exp\left(
-C\exp(C\,\Theta(R)^{m})\right) $ with $C$ structural.
\end{theorem}

\begin{theorem}
[abstract continuity]\label{th-abs-cont}Assume (H1)--(H6) at $(x,R_{0})$. Let
$u\in W_{A}^{1,2}(\Omega)$ be a weak solution of (\ref{eq-quasi}) in $\Omega$.
Then%
\[
\lim_{\rho\rightarrow0}\operatorname*{ess\,osc}_{B(x,\rho)}u=0.
\]
More precisely, fix $0<R<R_{0}$ and let $\tau=1-2\delta(R)/R\in\lbrack\frac
{1}{2},1)$, $\rho_{j}=\tau^{j}R$. Then for every $J\geq1$,%
\begin{equation}
\operatorname*{ess\,osc}_{B(x,\rho_{J})}u\leq e^{-\frac{1}{2}\sum
_{j<J}\varepsilon_{0}(\rho_{j})}\operatorname*{ess\,osc}_{B(x,R)}u+\sum
_{j<J}\mathcal{N}_{x}(\rho_{j})\,e^{-\frac{1}{2}\sum_{j<i<J}\varepsilon
_{0}(\rho_{i})}.\label{modulus}%
\end{equation}
Here $\operatorname*{ess\,osc}_{B(x,R)}u<\infty$, and $\sum_{j}\varepsilon
_{0}(\rho_{j})=\infty$ while $\sum_{j}\mathcal{N}_{x}(\rho_{j})<\infty$.
\end{theorem}

\begin{corollary}
\label{cor-rep}Suppose that (H1)--(H6) hold at every $x\in\Omega$ (with
$R_{0}$ and the other data depending on $x$). Then every weak solution $u\in
W_{A}^{1,2}(\Omega)$ of (\ref{eq-quasi}) has a representative that is
continuous in $\Omega$.
\end{corollary}

\begin{proof}
For $y\in\Omega$ let $\tilde{u}(y)=\lim_{\rho\rightarrow0}%
\operatorname*{ess\,sup}_{B(y,\rho)}u$. The limit exists and is finite by
Theorem \ref{th-abs-cont}. Moreover $\operatorname*{ess\,inf}_{B(y,\rho)}%
u\leq\tilde{u}(y)\leq\operatorname*{ess\,sup}_{B(y,\rho)}u$ for every $\rho$.

\emph{$\tilde{u}$ is continuous.} Let $y\in B(x,\rho/2)$. Then $B(y,s)\subset
B(x,\rho)$ for $s<\rho/2$, so $\tilde{u}(y)$ and $\tilde{u}(x)$ both lie in
$[\operatorname*{ess\,inf}_{B(x,\rho)}u,\operatorname*{ess\,sup}_{B(x,\rho
)}u]$, and $|\tilde{u}(y)-\tilde{u}(x)|\leq\operatorname*{ess\,osc}%
_{B(x,\rho)}u$. Since $d$ is continuous, $B(x,\rho/2)$ is a Euclidean
neighbourhood of $x$, so $\tilde{u}$ is continuous at $x$.

\emph{$\tilde{u}=u$ a.e.} Since metric balls are Euclidean-open, every
Euclidean Lebesgue point $y$ of $u$ satisfies $u(y)\in\lbrack
\operatorname*{ess\,inf}_{B(y,\rho)}u,\operatorname*{ess\,sup}_{B(y,\rho)}u]$
for all $\rho$. Hence $u(y)=\tilde{u}(y)$.
\end{proof}

\subsection{Roadmap}

The proofs use the following chain.

\begin{enumerate}
\item \emph{Local estimates}, Section \ref{sec-local}.

\begin{itemize}
\item From Part I: $L^{2}$--$L^{\infty}$ bounds for subsolutions (Theorem
\ref{th-PartI-sub}) and for negative powers of nonnegative supersolutions
(Theorem \ref{L_infinity-beta}). These use (H1), (H2), (H4).

\item The weak $L^{1}$ estimate for $\log(u+k)$ (Lemma \ref{lem weak1}). This
uses (H1), (H2), (H3) and the deduction $\delta$.
\end{itemize}

\item \emph{Bombieri--Giusti lemma}, Section \ref{sec-BG}. This is a purely
real-variable statement.

\item \emph{Lower half-Harnack inequality}, Section \ref{sec-lowH}. Theorem
\ref{th-lowH} follows from 1 and 2.

\item \emph{Measure-to-pointwise lemma and oscillation estimate}, Section
\ref{sec-lowH}. Lemma \ref{lem-meas} combines Theorem \ref{th-lowH} with the
other half of Lemma \ref{lem weak1}. Applying it to $u-\inf u$ or $\sup u-u$
gives Theorem \ref{th-osc}.

\item \emph{Continuity}, Section \ref{sec-cont}. Iterate Theorem \ref{th-osc}.
The growth condition (H5) makes $\sum\varepsilon_{0}=\infty$, and Dini
admissibility (H6) makes $\sum\mathcal{N}<\infty$.
\end{enumerate}

\begin{remark}
[on a full Harnack inequality]\label{rem-upper}A two-sided Harnack inequality
would also require an upper estimate $\operatorname*{ess\,sup}_{B(x,R_{2}%
)}(u+k)\leq b^{\prime}(R)\exp%
%TCIMACRO{\TeXButton{dashint}{\dashint}}%
%BeginExpansion
\dashint
%EndExpansion
\log(u+k)$. Via Bombieri--Giusti, this needs bounds $\sup(u+k)^{\theta}\leq C(%
%TCIMACRO{\TeXButton{dashint}{\dashint}}%
%BeginExpansion
\dashint
%EndExpansion
(u+k)^{2\theta})^{1/2}$ for subsolutions, \emph{uniformly} in small $\theta
>0$, with constants of the form in (\ref{BG1}) below.

\begin{itemize}
\item Interpolating Theorem \ref{th-PartI-sub} (as in Lemma \ref{lem-interp})
gives such bounds for each fixed $\theta$, but with $\theta$-dependent
constants that the Bombieri--Giusti argument cannot absorb.

\item The alternative route through reverse H\"{o}lder inequalities for small
powers of supersolutions seems to require a power gain in the Sobolev
inequality, which the Orlicz gain $\Phi_{m}$ does not provide.
\end{itemize}

We leave the full Harnack inequality in this setting as an open problem. By a
\textquotedblleft full Harnack\textquotedblright\ we mean a Harnack-type
inequality with blow-up constants as the scale decreases. Such inequality may
provide an alternate path to continuity if the blow up has enough control. As
we see in this work, continuity does not \textquotedblleft
need\textquotedblright\ the two-sided Harnack inequality.
\end{remark}

\section{Local estimates\label{sec-local}}

\subsection{Estimates from Part I}

The following two theorems are proved in Part I \cite{KoRiSaSh25} by Moser
iteration based on the nonhomogeneous $\Phi_{m}$-Orlicz--Sobolev inequality.
We restate them in the notation of this paper.

\begin{theorem}
[{\cite[Theorem 4.1]{KoRiSaSh25}}]\label{th-PartI-sub}Assume (H1), (H2), (H4)
at $(x,R_{0})$. Let $0<r\leq R_{0}$, let $u$ be a weak subsolution of
(\ref{eq-linear}) in $B(x,r)$, and let $k\geq\mathcal{N}_{x}(r)$. Then for
$0<\nu<1$,%
\begin{equation}
\operatorname*{ess\,sup}_{B(x,\nu r)}(u^{+}+k)\leq\exp\left( C_{I}\left(
\ln\frac{\varphi(r)}{(1-\nu)r}\right) ^{m}\right) \left(
%TCIMACRO{\TeXButton{dashint}{\dashint}}%
%BeginExpansion
\dashint
%EndExpansion
_{B(x,r)}(u^{+}+k)^{2}\right) ^{1/2}.\label{L-inf-1}%
\end{equation}

\end{theorem}

\begin{theorem}
[{\cite[Theorem 4.2]{KoRiSaSh25}}]\label{L_infinity-beta}Assume (H1), (H2),
(H4) at $(x,R_{0})$. Let $0<r\leq R_{0}$, let $u\geq0$ be a weak supersolution
of (\ref{eq-linear}) in $B(x,r) $, and let $k\geq\mathcal{N}_{x}(r)$ with
$k>0$. Then for $0<\nu<1$ and $\beta<0$,%
\begin{equation}
\operatorname*{ess\,sup}_{B(x,\nu r)}(u+k)^{\beta}\leq\exp\left( C_{I}\left(
(|\beta|+1)^{m}+\left( \ln\frac{\varphi(r)}{(1-\nu)r}\right) ^{m}\right)
\right) \left(
%TCIMACRO{\TeXButton{dashint}{\dashint}}%
%BeginExpansion
\dashint
%EndExpansion
_{B(x,r)}(u+k)^{2\beta}\right) ^{1/2}.\label{L-inf-beta-est}%
\end{equation}

\end{theorem}

The constant $C_{I}$ depends only on $n,m,\lambda,\Lambda,C_{c}$. In Part I
these estimates are stated with $k=\mathcal{N}_{x}(r)$. The proofs use $k$
only through the inequality $\mathcal{N}_{x}(r)/k\leq1$, which is how the
right hand side is absorbed. So they hold verbatim for every $k\geq
\mathcal{N}_{x}(r)$, $k>0$.

\subsection{An interpolation lemma}

Theorem \ref{th-PartI-sub} controls the supremum by an $L^{2}$ average. The
following standard argument lowers the exponent. It is needed in Section
\ref{sec-MA}.

\begin{lemma}
\label{lem-interp}Let $0<\nu<1$, $r>0$, and let $w\geq0$ be measurable on
$B_{r}=B(x,r)$. Put $M(\rho)=\operatorname*{ess\,sup}_{B_{\rho}}w$ and assume
$M(\frac{1+\nu}{2}r)<\infty$. Suppose there are $K_{0}\geq1$, $\alpha>0$ and
$m\geq1$ such that%
\begin{equation}
M(\rho)\leq K_{0}\exp\left( \alpha\left( \ln\frac{r}{\rho^{\prime}-\rho
}\right) ^{m}\right) \left(
%TCIMACRO{\TeXButton{dashint}{\dashint}}%
%BeginExpansion
\dashint
%EndExpansion
_{B_{\rho^{\prime}}}w^{2}\right) ^{1/2}\qquad\text{for }\nu r\leq\rho
<\rho^{\prime}\leq r.\label{interp hyp}%
\end{equation}
Then for every $0<p<2$,%
\[
M(\nu r)\leq C_{\ast}\left(
%TCIMACRO{\TeXButton{dashint}{\dashint}}%
%BeginExpansion
\dashint
%EndExpansion
_{B_{r}}w^{p}\right) ^{1/p},
\]
where $C_{\ast}$ depends only on $p,\nu,K_{0},\alpha,m$ and $|B_{r}|/|B_{\nu
r}|$.
\end{lemma}

\begin{proof}
Let $a=%
%TCIMACRO{\TeXButton{dashint}{\dashint}}%
%BeginExpansion
\dashint
%EndExpansion
_{B_{r}}w^{p}$ and $\Gamma=|B_{r}|/|B_{\nu r}|$. For $\nu r\leq\rho^{\prime
}\leq r$,%
\[%
%TCIMACRO{\TeXButton{dashint}{\dashint}}%
%BeginExpansion
\dashint
%EndExpansion
_{B_{\rho^{\prime}}}w^{2}\leq M(\rho^{\prime})^{2-p}\,\Gamma a.
\]
Insert this in (\ref{interp hyp}) and apply Young's inequality with exponents
$\frac{2}{2-p}$ and $\frac{2}{p}$. This gives%
\[
M(\rho)\leq\frac{1}{2}M(\rho^{\prime})+C_{p}K(\rho,\rho^{\prime})^{\frac{2}%
{p}}(\Gamma a)^{\frac{1}{p}},
\]
where $K(\rho,\rho^{\prime})$ is the constant in (\ref{interp hyp}).

Let $\rho_{0}=\nu r$ and $\rho_{j+1}=\rho_{j}+\frac{3(1-\nu)r}{\pi
^{2}(j+1)^{2}}$, so that $\rho_{j}\nearrow\frac{1+\nu}{2}r$. Iterating, for
every $J$,%
\[
M(\nu r)\leq2^{-J}M\left( \tfrac{1+\nu}{2}r\right) +C_{p}(\Gamma a)^{\frac
{1}{p}}K_{0}^{\frac{2}{p}}\sum_{j\geq0}2^{-j}\exp\left( \frac{2\alpha}%
{p}\left( \ln\frac{\pi^{2}(j+1)^{2}}{3(1-\nu)}\right) ^{m}\right) .
\]
The series converges because $(\ln j)^{m}=o(j)$. Letting $J\rightarrow\infty$
proves the lemma.
\end{proof}

\begin{corollary}
[$L^{p}$--$L^{\infty}$ for subsolutions]\label{cor-Lp}Assume (H1), (H2), (H4)
at $(x,R_{0})$. Let $0<r\leq R_{0}$, let $u$ be a weak subsolution of
(\ref{eq-linear}) in $B(x,r)$, and let $k\geq\mathcal{N}_{x}(r)$. Then for
$\frac{1}{2}\leq\nu<1$ and $0<p<2$,%
\[
\operatorname*{ess\,sup}_{B(x,\nu r)}(u^{+}+k)\leq C_{\ast}\left(
%TCIMACRO{\TeXButton{dashint}{\dashint}}%
%BeginExpansion
\dashint
%EndExpansion
_{B(x,r)}(u^{+}+k)^{p}\right) ^{1/p},
\]
where $C_{\ast}$ depends only on structural constants, $p$, $\nu$, $\ln
\frac{\varphi(r)}{r}$ and $|B(x,r)|/|B(x,\nu r)|$.
\end{corollary}

\begin{proof}
Let $w=u^{+}+k$. By Theorem \ref{th-PartI-sub} with $\nu=\frac{1+\nu}{2}$,
$M(\frac{1+\nu}{2}r)<\infty$.

Let $\nu r\leq\rho<\rho^{\prime}\leq r$. Apply Theorem \ref{th-PartI-sub} on
$B(x,\rho^{\prime})$ with ratio $\rho/\rho^{\prime}$; this is allowed since
$k\geq\mathcal{N}_{x}(r)\geq\mathcal{N}_{x}(\rho^{\prime})$. Since
$\rho^{\prime}\geq r/2$, (\ref{superradius comp}) gives%
\[
\frac{\varphi(\rho^{\prime})}{\rho^{\prime}-\rho}=\frac{\varphi(\rho^{\prime
})}{\rho^{\prime}}\cdot\frac{\rho^{\prime}}{\rho^{\prime}-\rho}\leq
C_{\varphi}\frac{\varphi(r)}{r}\cdot\frac{r}{\rho^{\prime}-\rho}.
\]
Using $(a+b)^{m}\leq2^{m-1}(a^{m}+b^{m})$ for $a,b\geq0$, we obtain
(\ref{interp hyp}) with $K_{0}=\exp(C_{I}2^{m-1}(\ln\frac{C_{\varphi}%
\varphi(r)}{r})^{m})$ and $\alpha=C_{I}2^{m-1}$. Lemma \ref{lem-interp}
concludes the proof.
\end{proof}

\subsection{The weak $L^{1}$ estimate for the logarithm}

\begin{lemma}
\label{lem weak1}Assume (H1)--(H3) at $(x,R_{0})$, and let $\delta$ be a
$D$-deduction at $x$. Let $0<R\leq R_{0}$ and $R_{1}=R-\delta(R)$. Let $u\geq0
$ be a weak supersolution of (\ref{eq-linear}) in $B(x,R)$, let $k\geq
\mathcal{N}_{x}(R)$ with $k>0$, and put $\bar{u}=u+k$ and $a=%
%TCIMACRO{\TeXButton{dashint}{\dashint}}%
%BeginExpansion
\dashint
%EndExpansion
_{B(x,R_{1})}\log\bar{u}$. Then for all $s>0$,%
\begin{equation}
s\,\left| \{y\in B(x,R_{1}):\pm(\log\bar{u}(y)-a)>s\}\right| \leq C_{W}%
\frac{R_{1}}{\delta(R)}\,|B(x,R_{1})|,\label{weak_est}%
\end{equation}
for both choices of sign, where $C_{W}$ depends only on $\lambda,\Lambda
,C_{c},C_{P},D$.
\end{lemma}

\begin{proof}
Let $G=\log\bar{u}$. By (H2) there is $\psi\in\operatorname{Lip}_{\mathrm{c}%
}(B(x,R))$ with $0\leq\psi\leq1$, $\psi=1$ on $B(x,R_{1}+\frac{1}{2}%
\delta(R))$ and $|\nabla_{A}\psi|\leq2C_{c}/\delta(R)$.

\emph{Testing the equation.} The function $w=\psi^{2}/\bar{u}$ is nonnegative
and has compact support in $B(x,R)$. It belongs to $W_{A}^{1,2}$ by Lemma
\ref{lem-Friedrichs}(c)--(d), applied with $F(t)=1/(t+k)$ (Remark
\ref{rem-Friedrichs-cases}) and then with the Lipschitz function $\psi^{2}$.
The same lemma gives $\nabla_{A}w=2\psi\nabla_{A}\psi/\bar{u}-\psi^{2}%
\nabla_{A}G/\bar{u}$, where $G=\log\bar{u}$ and $\nabla_{A}G=\nabla_{A}%
u/\bar{u}$. Using $w$ in the supersolution inequality (Corollary
\ref{cor-test}) and using (\ref{admiss}),%
\[
\int\tilde{A}\nabla u\cdot\nabla w\geq-\mathcal{N}_{x}(R)\int|\nabla_{A}w|.
\]
Inserting the formula for $\nabla_{A}w$ and $\nabla_{A}u=\bar{u}\nabla_{A}G$
(with the convention for $\tilde{A}\nabla u\cdot\nabla w$ above),%
\begin{equation}
\int\psi^{2}\tilde{A}\nabla G\cdot\nabla G\leq2\int\psi\tilde{A}\nabla
G\cdot\nabla\psi+\mathcal{N}_{x}(R)\int|\nabla_{A}w|.\label{wlog-01}%
\end{equation}

\emph{Absorbing the right hand side.} Since $\mathcal{N}_{x}(R)/\bar{u}\leq1
$,%
\[
\mathcal{N}_{x}(R)\int|\nabla_{A}w|\leq2\int\psi|\nabla_{A}\psi|+\int\psi
^{2}|\nabla_{A}G|.
\]
By (\ref{struc_0}) and Cauchy--Schwarz,%
\[
2\int\psi\tilde{A}\nabla G\cdot\nabla\psi\leq2\Lambda\int\psi|\nabla
_{A}G||\nabla_{A}\psi|\leq\frac{\lambda}{4}\int\psi^{2}|\nabla_{A}G|^{2}%
+\frac{4\Lambda^{2}}{\lambda}\int|\nabla_{A}\psi|^{2},
\]
and also $\int\psi^{2}|\nabla_{A}G|\leq\frac{\lambda}{4}\int\psi^{2}%
|\nabla_{A}G|^{2}+\frac{1}{\lambda}|B(x,R)|$. The left side of (\ref{wlog-01})
is at least $\lambda\int\psi^{2}|\nabla_{A}G|^{2}$. Since $\delta(R)\leq
R\leq1$, absorbing gives%
\[
\int\psi^{2}|\nabla_{A}G|^{2}\leq C\frac{|B(x,R)|}{\delta(R)^{2}}\leq
CD\frac{|B(x,R_{1})|}{\delta(R)^{2}}.
\]
In particular, $G\in W_{A}^{1,2}(B(x,R_{1}+\frac{1}{2}\delta(R)))$ and by
Cauchy--Schwarz,%
\[
\int_{B(x,R_{1})}|\nabla_{A}G|\leq|B(x,R_{1})|^{1/2}\left(  \int\psi
^{2}|\nabla_{A}G|^{2}\right)  ^{1/2}\leq C\frac{|B(x,R_{1})|}{\delta(R)}.
\]
The Poincar\'{e} inequality (\ref{Poinc}) on $B(x,R_{1})$ and Chebyshev's
inequality give (\ref{weak_est}).
\end{proof}

Only the supersolution inequality was used. This is important in Section
\ref{sec-lowH}.

\section{A Bombieri--Giusti lemma\label{sec-BG}}

The following is a version of \cite[Lemma 3]{BoGi72} for concentric metric
balls, with explicit constants. No doubling property is involved.

\begin{lemma}
\label{lem bound}Let $B_{\rho}=B(x,\rho)$, $0<\nu_{0}<1$, and let $w>0$ be
measurable on $B_{\rho}$ with $\operatorname*{ess\,sup}_{B_{\rho}}w<\infty$.
Suppose there are constants $C_{1},C_{2}\geq1$, $\alpha>0$, $\tau\geq1$,
$\theta_{0}>1/(2e)$ and $a\in\mathbb{R}$ such that for every $s\in\lbrack
\nu_{0},1]$:

\begin{enumerate}
\item[(BG1)] for all $\nu\in\lbrack\nu_{0},1)$ and $\theta\in(0,\theta_{0}]$,%
\begin{equation}
\operatorname*{ess\,sup}_{B_{\nu s\rho}}w^{\theta}\leq C_{1}e^{\alpha(\ln
\frac{1}{1-\nu})^{\tau}}\left(
%TCIMACRO{\TeXButton{dashint}{\dashint}}%
%BeginExpansion
\dashint
%EndExpansion
_{B_{s\rho}}w^{2\theta}\right) ^{1/2};\label{BG1}%
\end{equation}

\item[(BG2)] for all $t>0$,%
\begin{equation}
t\,|\{y\in B_{s\rho}:\log w(y)>t+a\}|\leq C_{2}|B_{s\rho}|.\label{BG2}%
\end{equation}

\end{enumerate}

Then%
\begin{equation}
\operatorname*{ess\,sup}_{B_{\nu_{0}\rho}}w\leq b\,e^{a},\qquad b=\exp\left[
8S_{\alpha,\tau}C_{2}C_{1}^{4}\exp\left( 2^{\tau+1}\alpha\left( \ln\frac
{1}{1-\nu_{0}}\right) ^{\tau}\right) \right] ,\label{exp_bound'}%
\end{equation}
where $S_{\alpha,\tau}=\sum_{j\geq0}(\frac{3}{4})^{j}\exp\left( 2^{\tau
+1}\alpha\left( \ln\frac{\pi^{2}(j+1)^{2}}{6}\right) ^{\tau}\right) <\infty$.
\end{lemma}

\begin{proof}
For $s\in\lbrack\nu_{0},1]$ let $\Omega(s)=\operatorname*{ess}_{B_{s\rho}%
}(\log w-a)<\infty$, which is nondecreasing in $s$.

\emph{Step 1: a one-step estimate.} Let $\nu_{0}\leq s^{\prime}<s\leq1$ and
$\ell=\ln\frac{s}{s-s^{\prime}}$. We claim%
\begin{equation}
\Omega(s^{\prime})\leq\frac{3}{4}\Omega(s)+8C_{2}C_{1}^{4}e^{4\alpha\ell
^{\tau}}.\label{one-step}%
\end{equation}
If $\Omega(s)\leq0$, then $\Omega(s^{\prime})\leq\Omega(s)\leq\frac{3}%
{4}\Omega(s)$. If $0<\Omega(s)\leq2C_{2}$, then (\ref{one-step}) is trivial.
So let $\Omega=\Omega(s)>2C_{2}$.

\emph{Splitting.} Split $B_{s\rho}$ into $E_{1}=\{\log w-a>\Omega/2\}$ and its
complement. By (BG2) with $t=\Omega/2$, $|E_{1}|\leq\frac{2C_{2}}{\Omega
}|B_{s\rho}|$. Hence for $\theta>0$%
\begin{equation}
e^{-2\theta a}%
%TCIMACRO{\TeXButton{dashint}{\dashint}}%
%BeginExpansion
\dashint
%EndExpansion
_{B_{s\rho}}w^{2\theta}\leq\frac{2C_{2}}{\Omega}e^{2\theta\Omega}%
+e^{\theta\Omega}.\label{avr}%
\end{equation}

\emph{The choice of $\theta$.} Let $\theta=\frac{1}{\Omega}\ln\frac{\Omega
}{2C_{2}}>0$. We check that $\theta\leq\theta_{0}$. Otherwise $\ln\frac
{\Omega}{2}-\ln C_{2}>\theta_{0}\Omega$, so $2\theta_{0}\frac{\Omega}{2}%
-\ln\frac{\Omega}{2}<-\ln C_{2}\leq0$. This is impossible, because
$2\theta_{0}y>\ln y$ for all $y>0$ when $2\theta_{0}>1/e$.

\emph{Applying (BG1).} With this $\theta$, the two terms in (\ref{avr}) are
equal to $e^{\theta\Omega}$. Apply (BG1) with $\nu=s^{\prime}/s\in\lbrack
\nu_{0},1)$, so that $\ln\frac{1}{1-\nu}=\ell$. Then%
\[
\Omega(s^{\prime})\leq\frac{1}{\theta}\left( \ln C_{1}+\alpha\ell^{\tau
}\right) +\frac{1}{2\theta}\ln\left( 2e^{\theta\Omega}\right) =\frac{\Omega
}{2}+\frac{\Omega}{\ln\frac{\Omega}{2C_{2}}}Q,\qquad Q=\ln\left( \sqrt{2}%
C_{1}e^{\alpha\ell^{\tau}}\right) .
\]
If $\ln\frac{\Omega}{2C_{2}}\geq4Q$, then $\Omega(s^{\prime})\leq\frac{3}%
{4}\Omega$. Otherwise $\Omega<2C_{2}e^{4Q}=8C_{2}C_{1}^{4}e^{4\alpha\ell
^{\tau}}$, and $\Omega(s^{\prime})\leq\Omega$. In all cases (\ref{one-step}) holds.

\emph{Step 2: iteration.} Let $s_{0}=\nu_{0}$ and $s_{j+1}=s_{j}+(1-\nu
_{0})\frac{6}{\pi^{2}(j+1)^{2}}$, so that $s_{j}\nearrow1$. Then
$s_{j}/s_{j+1}\geq\nu_{0}$ automatically, and%
\[
\ln\frac{s_{j+1}}{s_{j+1}-s_{j}}\leq\ln\frac{1}{1-\nu_{0}}+\ln\frac{\pi
^{2}(j+1)^{2}}{6}.
\]
By (\ref{one-step}) and $(a+b)^{\tau}\leq2^{\tau-1}(a^{\tau}+b^{\tau})$,%
\[
\Omega(s_{j})\leq\frac{3}{4}\Omega(s_{j+1})+8C_{2}C_{1}^{4}\exp\left(
2^{\tau+1}\alpha\left(  \ln\frac{1}{1-\nu_{0}}\right)  ^{\tau}\right)
\exp\left(  2^{\tau+1}\alpha\left(  \ln\frac{\pi^{2}(j+1)^{2}}{6}\right)
^{\tau}\right)  .
\]
Iterating $J$ times,%
\begin{align*}
\Omega(\nu_{0})  & \leq\left(  \frac{3}{4}\right)  ^{J}\Omega(s_{J}%
)+8C_{2}C_{1}^{4}\exp\left(  2^{\tau+1}\alpha\left(  \ln\frac{1}{1-\nu_{0}%
}\right)  ^{\tau}\right)  \sum_{j=0}^{J-1}\left(  \frac{3}{4}\right)  ^{j}%
\exp\left(  2^{\tau+1}\alpha\left(  \ln\frac{\pi^{2}\left(  j+1\right)  ^{2}%
}{6}\right)  ^{\tau}\right) \\
& \leq\left(  \frac{3}{4}\right)  ^{J}\Omega(1)^{+}+8S_{\alpha,\tau}C_{2}%
C_{1}^{4}\exp\left(  2^{\tau+1}\alpha\left(  \ln\frac{1}{1-\nu_{0}}\right)
^{\tau}\right)
\end{align*}
where%
\[
S_{\alpha,\tau}=\sum_{j=0}^{\infty}\left(  \frac{3}{4}\right)  ^{j}\exp\left(
2^{\tau+1}\alpha\left(  \ln\frac{\pi^{2}\left(  j+1\right)  ^{2}}{6}\right)
^{\tau}\right)  =\sum_{j=0}^{\infty}\exp\left(  2^{\tau+1}\alpha\left(
\ln\frac{\pi^{2}\left(  j+1\right)  ^{2}}{6}\right)  ^{\tau}-j\ln\frac{4}%
{3}\right)  <\infty.
\]
Let $J\rightarrow\infty$ and use $\Omega(1)<\infty$.
\end{proof}

\begin{remark}
Compared with \cite{BoGi72}, the only new point is the bookkeeping of $C_{1}$
and of the dependence on $\nu$ in (BG1), which here is $\exp(\alpha(\ln
\frac{1}{1-\nu})^{\tau})$ rather than a power of $(1-\nu)^{-1}$. A constant of
the form $\exp(c(1-\nu)^{-q})$ with $q\geq1$ could not be used. In Step 2 the
gaps would have to satisfy $s_{j+1}-s_{j}\gtrsim j^{-1/q} $, which is not summable.
\end{remark}

\section{Lower half-Harnack, measure-to-pointwise, oscillation\label{sec-lowH}%
}

Throughout this section (H1)--(H4) hold at $(x,R_{0})$, $\delta$ is a
$D$-deduction at $x$, $0<R\leq R_{0}$, the radii $R_{1},R_{2}$ are as in
(\ref{radii}), and $B_{\rho}=B(x,\rho)$.

\subsection{Proof of Theorem \ref{th-lowH}}

Let $\bar{u}=u+k$ and $w=1/\bar{u}$. Then $0<w\leq1/k$. We apply Lemma
\ref{lem bound} on $B_{R_{1}}$ with $\nu_{0}=R_{2}/R_{1}$, so that $1-\nu
_{0}=\delta(R_{1})/R_{1}$ and $\nu_{0}\geq\frac{3}{4}$.

\emph{Verification of (BG1).} Let $s\in\lbrack\nu_{0},1]$, $\nu\in\lbrack
\nu_{0},1)$ and $\theta\in(0,1]$. Put $\rho=sR_{1}$, so $\rho\in\lbrack
\frac{1}{2}R_{1},R_{1}]$. Apply Theorem \ref{L_infinity-beta} on $B_{\rho}$
with $\beta=-\theta$; this is allowed since $k\geq\mathcal{N}_{x}%
(R)\geq\mathcal{N}_{x}(\rho)$. By (\ref{superradius comp}), $\varphi
(\rho)/\rho\leq C_{\varphi}\varphi(R_{1})/R_{1}$. Using $(a+b)^{m}\leq
2^{m-1}(a^{m}+b^{m})$ we get (BG1) with%
\[
C_{1}=\exp\left( C_{I}2^{m}+C_{I}2^{m-1}\left( \ln\frac{C_{\varphi}%
\varphi(R_{1})}{R_{1}}\right) ^{m}\right) ,\qquad\alpha=C_{I}2^{m-1}%
,\qquad\tau=m,\qquad\theta_{0}=1.
\]

\emph{Verification of (BG2).} Let $a=%
%TCIMACRO{\TeXButton{dashint}{\dashint}}%
%BeginExpansion
\dashint
%EndExpansion
_{B_{R_{1}}}\log w=-%
%TCIMACRO{\TeXButton{dashint}{\dashint}}%
%BeginExpansion
\dashint
%EndExpansion
_{B_{R_{1}}}\log\bar{u}$. Lemma \ref{lem weak1} (with the sign ``$-$'') gives,
for $t>0$,%
\[
t\,|\{y\in B_{R_{1}}:\log w>t+a\}|\leq C_{W}\frac{R_{1}}{\delta(R)}|B_{R_{1}%
}|.
\]
For $s\in\lbrack\nu_{0},1]$ we have $B_{sR_{1}}\subset B_{R_{1}}$ and, by
(D1), $|B_{R_{1}}|\leq D|B_{R_{2}}|\leq D|B_{sR_{1}}|$. Hence (BG2) holds with
$C_{2}=\max\{1,DC_{W}R_{1}/\delta(R)\}$.

\emph{Conclusion.} Lemma \ref{lem bound} gives $\operatorname*{ess\,sup}%
_{B_{R_{2}}}w\leq be^{a}$, which is (\ref{lowH}), with%
\[
b(R)=\exp\left[ 8S_{\alpha,m}C_{2}C_{1}^{4}\exp\left( 2^{m+1}\alpha\left(
\ln\frac{R_{1}}{\delta(R_{1})}\right) ^{m}\right) \right] .
\]

\emph{Simplifying the constant.} Recall $\Theta=\Theta(R)=\ln\frac{\varphi
(R)}{\delta(R)}\geq\ln4>1$.

\begin{itemize}
\item By (D2), $\delta(R_{1})\geq c_{\delta}\delta(R)$, so $\ln\frac{R_{1}%
}{\delta(R_{1})}\leq\Theta+\ln\frac{1}{c_{\delta}}$.

\item By (\ref{superradius comp}) and $\delta(R)\leq R$, $\ln\frac{C_{\varphi
}\varphi(R_{1})}{R_{1}}\leq\Theta+2\ln C_{\varphi}$.

\item $R_{1}/\delta(R)\leq e^{\Theta}\leq e^{\Theta^{m}}$.
\end{itemize}

Since $(\Theta+c)^{m}\leq(1+c)^{m}\Theta^{m}$ for $\Theta\geq1$, all the terms
combine to $b(R)\leq\exp[C\exp(C\Theta^{m})]$ with $C$ structural.
\hfill$\square$

\subsection{The measure-to-pointwise lemma}

The following lemma is where the De Giorgi idea enters. A supersolution that
is large on a fixed fraction of a ball is bounded below on a smaller ball. In
the uniformly elliptic case this follows from Moser's weak Harnack inequality
\cite{Moser61}, \cite[Theorem 8.18]{GiTr}. It is the elliptic form of Landis'
growth lemma \cite{Lan98} and of the expansion of positivity of
DiBenedetto--Gianazza--Vespri \cite[Chapter 4]{DiGiVe12}; see also
\cite[Section 3.4]{SaWh4}, \cite[Section 3]{Kas07}, and \cite{BrMo25}. Here we
derive it from the lower half-Harnack inequality and the other half of the
weak $\log$ estimate.

\begin{lemma}
\label{lem-meas}Let $v\geq0$ be a weak supersolution of $L_{\tilde{A}}%
v=\phi_{0}-\operatorname{div}_{A}\vec{\phi}_{1}$ in $B_{R}$, and let
$k\geq\mathcal{N}_{x}(R)$ with $k>0$. If%
\begin{equation}
|\{y\in B_{R_{1}}:v(y)\geq1\}|\geq\frac{1}{2}|B_{R_{1}}|,\label{meas hyp}%
\end{equation}
then%
\begin{equation}
\operatorname*{ess\,inf}_{B_{R_{2}}}(v+k)\geq\varepsilon_{0}(R):=\frac
{1}{b(R)}\exp\left( -2C_{W}\frac{R_{1}}{\delta(R)}\right) ,\label{meas}%
\end{equation}
where $b(R)$ is as in Theorem \ref{th-lowH}. In particular $\varepsilon
_{0}(R)\geq\exp(-C\exp(C\Theta(R)^{m}))$ with $C$ structural.
\end{lemma}

\begin{proof}
Let $\bar{v}=v+k$ and $a_{+}=%
%TCIMACRO{\TeXButton{dashint}{\dashint}}%
%BeginExpansion
\dashint
%EndExpansion
_{B_{R_{1}}}\log\bar{v}$. We claim that $a_{+}\geq-2C_{W}R_{1}/\delta(R)$.

If $a_{+}\geq0$ there is nothing to prove. Otherwise let $0<\eta<-a_{+}$ and
$s=-a_{+}-\eta>0$. Since $\bar{v}\geq1$ on $\{v\geq1\}$,%
\[
\{v\geq1\}\cap B_{R_{1}}\subset\{y\in B_{R_{1}}:\log\bar{v}>-\eta\}=\{y\in
B_{R_{1}}:\log\bar{v}-a_{+}>s\}.
\]
By (\ref{meas hyp}) and Lemma \ref{lem weak1} (sign ``$+$''),%
\[
\frac{s}{2}|B_{R_{1}}|\leq C_{W}\frac{R_{1}}{\delta(R)}|B_{R_{1}}|.
\]
Hence $-a_{+}-\eta\leq2C_{W}R_{1}/\delta(R)$. Let $\eta\rightarrow0$ to prove
the claim.

Theorem \ref{th-lowH} now gives $\operatorname*{ess\,inf}_{B_{R_{2}}}\bar
{v}\geq b(R)^{-1}e^{a_{+}}\geq\varepsilon_{0}(R)$.

The lower bound for $\varepsilon_{0}$ follows from Theorem \ref{th-lowH} and
$2C_{W}R_{1}/\delta(R)\leq2C_{W}e^{\Theta}\leq2C_{W}e^{\Theta^{m}}$.
\end{proof}

\subsection{Proof of Theorem \ref{th-osc}}

Freeze $\tilde{A}=\mathcal{A}(\cdot,u(\cdot))$, so that $u$ solves the linear
equation (\ref{eq-linear}). Let $M=\operatorname*{ess\,sup}_{B_{R}}u$,
$\mu=\operatorname*{ess\,inf}_{B_{R}}u$ and $\omega=M-\mu$. If $\omega=0$
there is nothing to prove.

The sets $\{u-\mu\geq\omega/2\}$ and $\{M-u\geq\omega/2\}$ cover $B_{R_{1}}$
up to a null set. So one of them has measure at least $\frac{1}{2}|B_{R_{1}}|$.

\emph{Case 1:} $|\{u-\mu\geq\omega/2\}\cap B_{R_{1}}|\geq\frac{1}{2}|B_{R_{1}%
}|$. Let $v=\frac{2}{\omega}(u-\mu)$. Then $v\geq0$ in $B_{R}$, $v$ is a weak
solution of $L_{\tilde{A}}v=\frac{2}{\omega}(\phi_{0}-\operatorname{div}%
_{A}\vec{\phi}_{1})$, whose right hand side has norm $\frac{2}{\omega
}\mathcal{N}_{x}(R)$, and $v$ satisfies (\ref{meas hyp}). Apply Lemma
\ref{lem-meas} with $k=\max\{\frac{2}{\omega}\mathcal{N}_{x}(R),\eta\}$,
$\eta>0$, and then let $\eta\rightarrow0$. Since $\varepsilon_{0}(R)$ does not
depend on the right hand side,%
\[
\operatorname*{ess\,inf}_{B_{R_{2}}}v\geq\varepsilon_{0}(R)-\frac{2}{\omega
}\mathcal{N}_{x}(R), \qquad\text{i.e.}\qquad\operatorname*{ess\,inf}%
_{B_{R_{2}}}u\geq\mu+\frac{\omega}{2}\varepsilon_{0}(R)-\mathcal{N}_{x}(R).
\]
Together with $\operatorname*{ess\,sup}_{B_{R_{2}}}u\leq M$, this gives
(\ref{osc}).

\emph{Case 2:} $|\{M-u\geq\omega/2\}\cap B_{R_{1}}|\geq\frac{1}{2}|B_{R_{1}}%
|$. Apply the same argument to $v=\frac{2}{\omega}(M-u)$. It solves the
equation with right hand side $-\frac{2}{\omega}(\phi_{0}-\operatorname{div}%
_{A}\vec{\phi}_{1})$, which has the same norm, and gives
$\operatorname*{ess\,sup}_{B_{R_{2}}}u\leq M-\frac{\omega}{2}\varepsilon
_{0}(R)+\mathcal{N}_{x}(R)$. \hfill$\square$

\begin{remark}
The proof uses only the \emph{supersolution} property of $u-\mu$ and $M-u$,
together with the boundedness of $u$. In particular no bound on small positive
powers of subsolutions is needed (compare Remark \ref{rem-upper}).
\end{remark}

\section{Proof of the abstract continuity theorem\label{sec-cont}}

\begin{proof}
[Proof of Theorem \ref{th-abs-cont}]\emph{Step 1: boundedness.} Let
$0<R<R_{0}$. Then $u$ and $-u$ are weak subsolutions of linear equations in
$B(x,R_{0})$, with right hand sides of norm $\mathcal{N}_{x}(R_{0})<\infty$.
Theorem \ref{th-PartI-sub} with $\nu=R/R_{0}$ shows that $u\in L^{\infty
}(B(x,R))$. So $\operatorname*{ess\,osc}_{B(x,\rho)}u<\infty$ for $\rho\leq R$.

\emph{Step 2: nested balls.} Let $\tau=1-2\delta(R)/R$; by (D1), $\tau
\in\lbrack\frac{1}{2},1)$. Let $\rho\leq R$, and let $\rho_{2}$ denote the
radius $R_{2}$ of (\ref{radii}) with $R$ replaced by $\rho$. By
(\ref{R2 lower}) and (D3),%
\[
\rho_{2}\geq\left( 1-2\frac{\delta(\rho)}{\rho}\right) \rho\geq\left(
1-2\frac{\delta(R)}{R}\right) \rho=\tau\rho.
\]
Hence, with $\rho_{j}=\tau^{j}R$, we have $B(x,\rho_{j+1})\subset
B(x,(\rho_{j})_{2})$.

\emph{Step 3: iteration.} Theorem \ref{th-osc} at the radius $\rho_{j}$ gives%
\[
\omega_{j+1}\leq\gamma_{j}\omega_{j}+\mathcal{N}_{x}(\rho_{j}),\qquad
\omega_{j}=\operatorname*{ess\,osc}_{B(x,\rho_{j})}u,\quad\gamma_{j}%
=1-\frac{\varepsilon_{0}(\rho_{j})}{2}\leq e^{-\varepsilon_{0}(\rho_{j})/2}.
\]
Iterating gives (\ref{modulus}).

\emph{Step 4: the two series.} By Theorem \ref{th-osc}, $\varepsilon_{0}%
(\rho_{j})\geq\exp(-C\exp(C\Theta(\rho_{j})^{m}))$ with $C$ structural. The
growth condition (\ref{delta_growth}), with this $C$, the ratio $\tau$ and the
radius $R$, gives $\sum_{j}\varepsilon_{0}(\rho_{j})=\infty$.

Each interval $(\tau^{i+1}R_{0},\tau^{i}R_{0}]$ contains at most one point of
$\{\rho_{j}\}$, and $\mathcal{N}_{x}$ is nondecreasing. So by (H6) and the
remark after Definition \ref{def Dini A adm},%
\[
\sum_{j}\mathcal{N}_{x}(\rho_{j})\leq\sum_{i}\mathcal{N}_{x}(\tau^{i}%
R_{0})<\infty.
\]

\emph{Step 5: conclusion.} Let $\varepsilon>0$ and choose $j_{0}$ with
$\sum_{j\geq j_{0}}\mathcal{N}_{x}(\rho_{j})<\varepsilon$. In (\ref{modulus}),
the first term and the finitely many terms with $j<j_{0}$ tend to $0$ as
$J\rightarrow\infty$, because $\sum_{i}\varepsilon_{0}(\rho_{i})=\infty$. The
remaining terms are bounded by $\varepsilon$. Hence $\limsup_{J}\omega_{J}%
\leq\varepsilon$. Since $\varepsilon$ is arbitrary and $\rho\mapsto
\operatorname*{ess}_{B(x,\rho)}u$ is monotone, the theorem follows.
\end{proof}

\begin{remark}
The argument replaces the iteration lemma of De Giorgi type used in earlier
versions (cf. \cite[Lemma 8.23]{GiTr}, \cite[Lemma 63]{SaWh4}) by the explicit
estimate (\ref{modulus}). This also gives a modulus of continuity in terms of
$\varepsilon_{0}$ and the tails of $\sum\mathcal{N}_{x}(\rho_{j})$.
\end{remark}

\section{The geometries $F_{k,\sigma}$ and the proof of Theorem
\ref{th-geom-cont}\label{sec-geom}}

\subsection{Structural conditions and known facts}

We work in the plane with $A(x_{1},x_{2})=\operatorname{diag}(1,f(x_{1})^{2}%
)$, $f=e^{-F}$ even. Following \cite{KoRiSaSh19}, we assume that for some
$r_{F}>0 $, $C\geq1$ and $c_{F}>0$:

\begin{enumerate}
\item[(SC1)] $\lim_{r\rightarrow0^{+}}F(r)=+\infty$;

\item[(SC2)] $F^{\prime}<0$ and $F^{\prime\prime}>0$ on $(0,r_{F})$;

\item[(SC3)] $C^{-1}|F^{\prime}(r)|\leq|F^{\prime}(t)|\leq C|F^{\prime}(r)|$
for $\frac{r}{2}<t<2r<r_{F}$;

\item[(SC4)] $r\mapsto\frac{1}{-rF^{\prime}(r)}$ is increasing on $(0,r_{F})$
and bounded by $1/c_{F}$;

\item[(SC5)] $\frac{F^{\prime\prime}(r)}{-F^{\prime}(r)}\approx\frac{1}{r}$ on
$(0,r_{F})$.
\end{enumerate}

For $F=F_{k,\sigma}$ write $L_{j}=\ln^{(j)}\frac{1}{r}$. Then%
\begin{equation}
-rF_{k,\sigma}^{\prime}(r)=L_{k}^{\sigma}+\frac{\sigma L_{k}^{\sigma-1}}%
{L_{2}\cdots L_{k-1}}=L_{k}^{\sigma}(1+o(1))\qquad(r\rightarrow0),\label{rF'}%
\end{equation}
with the empty product equal to $1$ when $k=2$. The function $-rF^{\prime}$ is
slowly varying and decreasing. From this, (SC1)--(SC5) follow for
$0<r<r_{k,\sigma}$ small enough.

Let $d$ be the Carnot--Carath\'{e}odory metric of $A$. Along a subunit curve
$|\dot{x}_{1}|\leq1$, so for $P=(0,p_{2})$,%
\begin{equation}
B(P,\rho)\subset\{|x_{1}|<\rho\}.\label{ball strip}%
\end{equation}
Since $A$ does not depend on $x_{2}$, the metric is invariant under vertical
translations. Hence all statements at the origin hold at every point $P$ of
the axis $\{x_{1}=0\}$, with the same constants.

We use the following facts.

\begin{enumerate}
\item[(F1)] $d$ is topologically equivalent to the Euclidean metric
\cite[Chapter 7]{KoRiSaSh19}.

\item[(F2)] $|B(P,\rho)|\approx\frac{f(\rho)}{|F^{\prime}(\rho)|^{2}}$ for
$0<\rho\leq r_{F}^{\prime}$, with comparability constant $C_{B}$
\cite[Conclusion 45]{KoRiSaSh19}.

\item[(F3)] The $(1,1)$-Poincar\'{e} inequality holds at $(P,r_{F}^{\prime}) $
\cite[Lemma 67]{KoRiSaSh19}.

\item[(F4)] For $F=F_{k,\sigma}$, $k\geq2$, $\sigma>0$, $m>1$, the
nonhomogeneous $(\Phi_{m},\varphi)$-Orlicz--Sobolev inequality holds at
$(P,r_{m,\sigma})$ with $\varphi$ given by (\ref{phi-grow}) \cite[Corollary
5.7]{KoRiSaSh25}.
\end{enumerate}

\subsection{Verification of the hypotheses at points of the axis}

Fix $k\geq2$, $\sigma>0$, $m>2$, $F=F_{k,\sigma}$, and a point $P$ on the
axis. Let $R_{0}>0$ be small.

\begin{lemma}
\label{lem-geom-H}For $R_{0}$ small, (H1)--(H5) hold at $(P,R_{0})$ for every
Carath\'{e}odory $\mathcal{A}$ with $\mathcal{A}(\cdot,z)\in\mathfrak{A}%
(A,\Lambda,\lambda)$, with:

\begin{itemize}
\item the superradius $\varphi$ of (\ref{phi-grow});

\item the deduction $\delta(\rho)=c_{0}/|F^{\prime}(\rho)|$, where
$c_{0}=c_{F}/4$;
\end{itemize}

provided that either $k=3$ and $\sigma m(m-1)<1$, or $k\geq4$. When $k\geq2 $
is arbitrary, (H1)--(H4) and (D1)--(D3) still hold, i.e. everything except the
growth condition.
\end{lemma}

\begin{proof}
\emph{(H1)} $\sqrt{A}=\operatorname{diag}(1,f)$, and $f$ is Lipschitz (indeed
$f^{\prime}=-F^{\prime}e^{-F}\rightarrow0$ at $0$). $d$ is continuous by (F1).

\emph{(H2)} follows from (F1) and Lemma \ref{spec_cutoff_lemma}.

\emph{(H3)} is (F3).

\emph{(H4)} is (F4). It remains to check that $\varphi$ is a superradius.

\begin{itemize}
\item $\varphi(r)\geq r$ is clear.

\item $\varphi$ is increasing for small $r$: $\frac{d}{dr}\ln\frac{\varphi
(r)}{r}=-C_{m}\sigma(m-1)L_{k}^{\sigma(m-1)-1}/(rL_{1}\cdots L_{k-1})$, which
is $o(1/r)$.

\item (\ref{superradius comp}) holds: for $\frac{r}{2}\leq\rho\leq r$,
$L_{k}(\rho)-L_{k}(r)\rightarrow0$, so $L_{k}(\rho)^{\sigma(m-1)}%
-L_{k}(r)^{\sigma(m-1)}\rightarrow0$.
\end{itemize}

\emph{Deduction, (D1).} By (SC4), $\delta(\rho)/\rho=c_{0}/(-\rho F^{\prime
}(\rho))\leq c_{0}/c_{F}=\frac{1}{4}$. By (F2),%
\[
\frac{|B(P,\rho)|}{|B(P,\rho-\delta(\rho))|}\leq C_{B}^{2}\frac{f(\rho
)}{f(\rho-\delta(\rho))}\frac{|F^{\prime}(\rho-\delta(\rho))|^{2}}{|F^{\prime
}(\rho)|^{2}}.
\]
Here $f(\rho)/f(\rho-\delta)=\exp(\int_{\rho-\delta}^{\rho}|F^{\prime}%
|)\leq\exp(C\delta|F^{\prime}(\rho)|)=e^{Cc_{0}}$, by (SC3) on $[\rho/2,\rho
]$. Also $|F^{\prime}(\rho-\delta)|\leq C|F^{\prime}(\rho)|$. So (D1) holds
with $D=\max\{2,C_{B}^{2}C^{2}e^{Cc_{0}}\}$.

\emph{(D2)} follows from (SC3), with $c_{\delta}=1/C$.

\emph{(D3)} follows from (SC4).

\emph{Growth (H5).} By (\ref{phi-grow}) and (\ref{rF'}),%
\[
\Theta(\rho)=\ln\frac{\varphi(\rho)}{\rho}+\ln\frac{-\rho F^{\prime}(\rho
)}{c_{0}}\leq C_{m}L_{k}^{\sigma(m-1)}+\sigma L_{k+1}+C.
\]
Hence $\Theta(\rho)^{m}\leq C(L_{k}^{\sigma m(m-1)}+L_{k+1}^{m}+1)$.

\begin{itemize}
\item If $k=3$ and $\sigma m(m-1)<1$, then $L_{3}^{\sigma m(m-1)}=o(L_{3})$
and $L_{4}^{m}=o(L_{3})$.

\item If $k\geq4$, every power of $L_{k}$ or $L_{k+1}$ is $o(L_{3})$.
\end{itemize}

In both cases (\ref{growth ln3}) holds, and Lemma \ref{lem-growth-suff} gives
the growth condition.
\end{proof}

\begin{lemma}
[Dini admissibility]\label{lem-geom-Dini}Let $k\geq2$, $\sigma>0$, $m>1$, and
let $K$ be a compact neighbourhood of $P$ in $\Omega$.

\begin{enumerate}
\item If $\phi_{0}\in L^{\Phi_{m}^{\ast}}(K)$, then $(\phi_{0},0)$ is Dini
admissible at $(P,R_{0})$ for $R_{0}$ small.

\item If $|\vec{\phi}_{1}(x)|\leq\eta(|x_{1}|)$ for a.e. $x\in K$, with $\eta$
nondecreasing and $\int_{0}^{1}\eta(t)\frac{dt}{t}<\infty$, then $(0,\vec
{\phi}_{1})$ is Dini admissible at $(P,R_{0})$ for $R_{0}$ small.
\end{enumerate}
\end{lemma}

\begin{proof}
Take $R_{0}$ so small that $B(P,R_{0})\subset K$ (possible by (F1)).

(2) By (\ref{drift bound}) and (\ref{ball strip}), $\mathcal{N}^{1}%
(B(P,\rho))\leq\eta(\rho)$. For $\tau\in(0,1)$ and $j\geq1$, monotonicity
gives $\eta(\tau^{j}R_{0})\leq(\ln\frac{1}{\tau})^{-1}\int_{\tau^{j}R_{0}%
}^{\tau^{j-1}R_{0}}\eta(t)\frac{dt}{t}$. Summing over $j$ gives a finite sum.

(1) We verify the superradius summability condition (\ref{super-summ}) and
apply Proposition \ref{Dini-suff}. Here $\Phi_{m}$ is submultiplicative, and
(H4) holds by Lemma \ref{lem-geom-H}. Let $\rho_{j}=\tau^{j}R_{0}$ and
$L_{i}=L_{i}(\rho_{j})$.

\begin{itemize}
\item \emph{Volume ratio.} By (F2) and (\ref{rF'}), for $j$ large,%
\[
\ln\frac{|B(P,R_{0})|}{|B(P,\rho_{j})|}\leq F(\rho_{j})+3\ln\frac{1}{\rho_{j}%
}\leq2L_{1}L_{k}^{\sigma}.
\]

\item \emph{Conjugate function.} By Lemma \ref{lem-Phistar} in the Appendix,
$(\Phi_{m}^{\ast})^{-1}(T)\leq2\exp(c_{m}(\ln T)^{\frac{m-1}{m}})$ for large
$T$.

\item \emph{Superradius.} By (\ref{phi-grow}), $\varphi(\rho_{j})\leq\tau
^{j}\exp(C_{m}L_{k}^{\sigma(m-1)})$.

\item \emph{Size of the logarithms.} $L_{1}\leq(j+1)L$ with $L=\max\{\ln
\frac{1}{\tau},\ln\frac{1}{R_{0}}\}$, and $L_{k}\leq L_{2}\leq\ln((j+1)L) $
since $k\geq2$.
\end{itemize}

Hence the $j$-th term of (\ref{super-summ}) is at most%
\[
2\exp\left( -j\ln\frac{1}{\tau}+C_{m}(\ln((j+1)L))^{\sigma(m-1)}+c_{m}\left(
2(j+1)L\right) ^{\frac{m-1}{m}}(\ln((j+1)L))^{\frac{\sigma(m-1)}{m}}\right) .
\]
This is summable because $\frac{m-1}{m}<1$.
\end{proof}

\subsection{Proof of Theorem \ref{th-geom-cont}}

Let $u\in W_{A}^{1,2}(\Omega)$ be a weak solution and $x\in\Omega$.

\emph{Points off the axis.} If $x_{1}\neq0$, the equation is uniformly
elliptic near $x$. Indeed $f\geq c>0$ near $x$, and $\tilde{A}\approx A$ by
(\ref{struc_0}). Moreover:

\begin{itemize}
\item $\phi_{0}\in L_{\mathrm{loc}}^{q}$ for every $q<\infty$, since $\Phi
_{m}^{\ast}$ grows faster than every power (Lemma \ref{lem-Phistar});

\item $\sqrt{A}\vec{\phi}_{1}\in L_{\mathrm{loc}}^{\infty}$.
\end{itemize}

So $u$ is H\"{o}lder continuous near $x$ by the De Giorgi--Nash--Moser theory
\cite[Theorem 8.24]{GiTr}. The equation $-\operatorname{div}\tilde{A}\nabla
u=\phi_{0}-\operatorname{div}(\sqrt{A}\vec{\phi}_{1})$ is of the form treated there.

\emph{Points on the axis.} If $x=P=(0,p_{2})$, choose $R_{0}$ small, with
$R_{0}\leq r_{\ast}$ so that $f=f_{k,\sigma}$ on the relevant strip, and take
$m$ as in the theorem.

\begin{itemize}
\item By Lemma \ref{lem-geom-H}, (H1)--(H5) hold at $(P,R_{0})$.

\item By Lemma \ref{lem-geom-Dini} and assumption (\ref{drift Dini}), (H6)
holds, since $\mathcal{N}_{P}=\mathcal{N}^{0}+\mathcal{N}^{1}$.
\end{itemize}

Theorem \ref{th-abs-cont} gives $\operatorname*{ess\,osc}_{B(P,\rho
)}u\rightarrow0$.

\emph{Continuous representative.} Thus the hypotheses of Theorem
\ref{th-abs-cont} hold at every point of the axis. At the other points the
oscillation over Euclidean balls, hence over metric balls by (F1), tends to
zero. The proof of Corollary \ref{cor-rep} only uses $\operatorname*{ess\,osc}%
_{B(y,\rho)}u\rightarrow0$ at every $y$, so it applies and gives a continuous
representative. \hfill$\square$

\section{Application to the Monge--Amp\`{e}re equation\label{sec-MA}}

\subsection{The density}

Let $0<\sigma<\frac{1}{2}$ and $0<R<e^{-e}$; the latter guarantees $\ln
^{(3)}\frac{1}{|x|}>0$ for $0<|x|<R$. Let%
\begin{equation}
\kappa(x)=f_{3,\sigma}(x)^{2}=e^{-2F_{3,\sigma}(|x|)},\qquad0<|x|\leq
R,\qquad\kappa(0)=0.\label{MA-k}%
\end{equation}
Then $\kappa$ is even, positive for $x\neq0$, and vanishes to infinite order
at $0$. Let $L_{j}=\ln^{(j)}\frac{1}{x}$ for $0<x\leq R$. Since $L_{3}\geq
\ln^{(3)}\frac{1}{R}>0$ there, (\ref{rF'}) gives $|F_{3,\sigma}^{\prime
}(x)|\leq C_{R}L_{3}^{\sigma}/x$. Hence%
\begin{equation}
\frac{|\kappa^{\prime}(x)|}{\sqrt{\kappa(x)}}=2|F_{3,\sigma}^{\prime
}(x)|e^{-F_{3,\sigma}(x)}\leq\frac{2C_{R}L_{3}^{\sigma}}{x}\,x^{L_{3}^{\sigma
}}.\label{MA-kprime}%
\end{equation}
The right side tends to $0$ faster than any power of $x$. In particular
$\kappa,\sqrt{\kappa}\in C^{1}([-R,R])$, $\kappa(0)=\kappa^{\prime}(0)=0$, and%
\begin{equation}
\eta_{\kappa}(t)=\sup_{0<|s|\leq t}\frac{|\kappa^{\prime}(s)|}{\sqrt
{\kappa(s)}}\label{eta kappa}%
\end{equation}
is nondecreasing, with $\eta_{\kappa}(t)\leq C_{R}^{\prime}t$ for $0<t\leq R$.

Recall Theorem \ref{th-MA}. The \emph{strict convexity} hypothesis is%
\[
\text{(H)}\qquad y\mapsto u(0,y)\text{ is strictly convex on }(-R,R).
\]

\begin{remark}
\label{rem-MA-H}Away from the line $x=0$ no hypothesis is needed. If $x\neq0
$, the eigenvalues $0\leq\lambda_{1}\leq\lambda_{2}\leq M=\Vert D^{2}%
u\Vert_{L^{\infty}(B)}$ of $D^{2}u$ satisfy $\lambda_{1}\lambda_{2}=\kappa
(x)$, hence $u_{yy}\geq\lambda_{1}\geq\kappa(x)/M>0$ a.e.

Hypothesis (H) holds if $u$ is strictly convex. Some hypothesis of this type
is unavoidable for the partial Legendre transform in $y$. Let $g$ solve
$2gg^{\prime\prime}-4(g^{\prime})^{2}=c>0$, $g(0)=1$, $g^{\prime}(0)=0$. Then
$u(x,y)=x^{2}g(y)$ is convex near the origin and $\det D^{2}u=cx^{2}$ depends
only on $x$. But $u(0,y)\equiv0$, and $(x,y)\mapsto(x,u_{y})$ collapses the
segment $\{0\}\times I$ to a point. This $u$ is smooth, so the obstruction is
to the method, not to the conclusion. See also \cite[Theorem 1.1 and
Proposition 1.2]{SaWh06}.
\end{remark}

\begin{remark}
\label{rem-MA-inhom}The partial Legendre transform turns the equation into a
\emph{linear} degenerate equation (\ref{MA-eq-v}), because $\kappa$ depends
only on $x$.

\begin{itemize}
\item Its $t$-derivative $v_{t}$ satisfies the homogeneous equation.

\item Its $s$-derivative $v_{s}$ satisfies an equation with a drift
$-\operatorname{div}_{A}\vec{\phi}_{1}$, where $|\vec{\phi}_{1}|\lesssim
\eta_{\kappa}(|s|) $ (Lemma \ref{lem-MA-vs}).
\end{itemize}

So the inhomogeneous theory, with the Dini condition (\ref{drift Dini}), is
genuinely needed even though the Monge--Amp\`{e}re equation itself has no
inhomogeneous terms.
\end{remark}

\subsection{Preliminaries}

Let $M=\Vert D^{2}u\Vert_{L^{\infty}(B)}$. Since $u$ is convex and $C^{1,1} $:

\begin{itemize}
\item $D^{2}u$ exists a.e. and $0\leq D^{2}u\leq MI$ a.e.;

\item $u_{xy}=\partial_{y}u_{x}=\partial_{x}u_{y}$ a.e.;

\item the equation $u_{xx}u_{yy}-u_{xy}^{2}=\kappa(x)$ holds pointwise a.e.
\end{itemize}

We use the identity%
\begin{equation}
\int_{B}h\,(g_{y}W_{x}-g_{x}W_{y})=\int_{B}g\,(h_{x}W_{y}-h_{y}W_{x}%
),\label{MA-null}%
\end{equation}
valid for $g\in W^{1,\infty}(B)$, $W\in\operatorname{Lip}_{\mathrm{c}}(B)$ and
$h\in C^{1}(\overline{B})$. For $W\in C_{c}^{\infty}$ it follows by
integrating both terms on the left by parts; the terms with $W_{xy}$ cancel.
The general case follows by mollifying $W$.

\subsection{The partial Legendre transform}

Define $\Gamma(x,y)=(s,t)=(x,u_{y}(x,y))$ on $B$.

\begin{lemma}
\label{lem-MA-Gamma}We have $u_{yy}>0$ a.e. in $B$, and for every $x$ the
function $y\mapsto u_{y}(x,y)$ is strictly increasing on $B_{x}=\{y:(x,y)\in
B\}$. Consequently:

\begin{itemize}
\item $\Gamma$ is a Lipschitz homeomorphism of $B$ onto an open set
$U=\Gamma(B)\subset(-R,R)\times\mathbb{R}$;

\item $\Gamma^{-1}(s,t)=(s,v(s,t))$, where $v(s,\cdot)$ is the continuous,
strictly increasing inverse of $u_{y}(s,\cdot)$;

\item for every nonnegative measurable $G$ on $U$,%
\begin{equation}
\int_{U}G\,ds\,dt=\int_{B}(G\circ\Gamma)\,u_{yy}\,dx\,dy;\label{MA-area}%
\end{equation}

\item a set $E\subset U$ is null if and only if $\Gamma^{-1}(E)$ is null.
\end{itemize}
\end{lemma}

\begin{proof}
By Remark \ref{rem-MA-H}, $u_{yy}>0$ a.e.

\emph{Strict monotonicity.} Fix $x_{0}\neq0$ and $y_{1}<y_{2}$ in $B_{x_{0}}
$. For a.e. $x$ near $x_{0}$, $u_{y}(x,\cdot)$ is Lipschitz with derivative
$u_{yy}(x,\cdot)\geq\kappa(x)/M$ a.e. So $u_{y}(x,y_{2})-u_{y}(x,y_{1}%
)\geq(y_{2}-y_{1})\min_{|x-x_{0}|\leq\eta}\kappa/M$. Letting $x\rightarrow
x_{0}$ gives $u_{y}(x_{0},y_{2})>u_{y}(x_{0},y_{1})$. For $x_{0}=0$, strict
monotonicity of $u_{y}(0,\cdot)$ is (H).

\emph{Consequences.} Thus $\Gamma$ is injective and Lipschitz, and by
invariance of domain it is a homeomorphism onto the open set $U$. Its Jacobian
is $u_{yy}$ a.e., so (\ref{MA-area}) is the area formula. If $|\Gamma
^{-1}(E)|=0$ then $|E|=0$, because $\Gamma$ is Lipschitz. Conversely, if
$|E|=0$ then $\int_{\Gamma^{-1}(E)}u_{yy}=0$, so $|\Gamma^{-1}(E)|=0$.
\end{proof}

By Lemma \ref{lem-MA-Gamma} the following functions are well defined a.e. in
$U$:%
\begin{equation}
a=\frac{1}{u_{yy}}\circ\Gamma^{-1},\qquad b=-\frac{u_{xy}}{u_{yy}}\circ
\Gamma^{-1}.\label{MA-ab}%
\end{equation}
In the $(s,t)$-plane let $A(s,t)=\operatorname{diag}(1,\kappa(s))$, so that
$\nabla_{A}=(\partial_{s},\sqrt{\kappa(s)}\partial_{t})$. Note that $\sqrt{A}$
is Lipschitz by (\ref{MA-kprime}).

\begin{lemma}
\label{lem-MA-v}The function $v$ belongs to $W_{\mathrm{loc}}^{1,1}(U)$, with
weak derivatives $v_{t}=a$ and $v_{s}=b$. Moreover, a.e. in $U$,%
\begin{equation}
v_{t}\geq\frac{1}{M},\qquad v_{s}^{2}\leq Mv_{t},\qquad\kappa(s)v_{t}^{2}\leq
Mv_{t},\label{MA-v-bounds}%
\end{equation}
and for every measurable $E\subset U$,%
\begin{equation}
\int_{E}v_{t}\,ds\,dt=|\Gamma^{-1}(E)|.\label{MA-vt-L1}%
\end{equation}
In particular $v\in W_{A,\mathrm{loc}}^{1,2}(U)$.
\end{lemma}

\begin{proof}
\emph{The bounds.} (\ref{MA-v-bounds}) follows from $u_{yy}\leq M$,
$u_{xy}^{2}\leq u_{xx}u_{yy}$ and the equation:%
\[
\frac{u_{xy}^{2}}{u_{yy}^{2}}\leq\frac{u_{xx}}{u_{yy}}\leq\frac{M}{u_{yy}},
\qquad\frac{\kappa}{u_{yy}^{2}}=\frac{u_{xx}u_{yy}-u_{xy}^{2}}{u_{yy}^{2}}%
\leq\frac{M}{u_{yy}}.
\]
(\ref{MA-vt-L1}) is (\ref{MA-area}) with $G=a\mathbf{1}_{E}$.

\emph{Pulling back test functions.} Let $w\in C_{c}^{\infty}(U)$ and
$W=w\circ\Gamma\in\operatorname{Lip}_{\mathrm{c}}(B)$. Then a.e.%
\begin{equation}
w_{t}\circ\Gamma=\frac{W_{y}}{u_{yy}},\qquad w_{s}\circ\Gamma=W_{x}%
-\frac{u_{xy}}{u_{yy}}W_{y}.\label{MA-pullback}%
\end{equation}

\emph{The derivative $v_{t}$.} Using (\ref{MA-area}), (\ref{MA-pullback}) and
$v\circ\Gamma=y$,%
\[
\int_{U}v\,w_{t}=\int_{B}y\,W_{y}=-\int_{B}W=-\int_{B}(a\circ\Gamma
)W\,u_{yy}=-\int_{U}a\,w.
\]

\emph{The derivative $v_{s}$.} Using (\ref{MA-null}) with $g=u_{y}$, $h=y$,%
\begin{align*}
\int_{U}v\,w_{s}  & =\int_{B}y\,(u_{yy}W_{x}-u_{xy}W_{y})=-\int_{B}u_{y}%
W_{x}=\int_{B}u_{xy}W\\
& =-\int_{B}(b\circ\Gamma)W\,u_{yy}=-\int_{U}b\,w.
\end{align*}
Since $v$ is continuous, the last statement follows from (\ref{MA-v-bounds}).
\end{proof}

\begin{lemma}
\label{lem-MA-eq}The function $v$ is a weak solution of%
\begin{equation}
L_{A}v=-\partial_{s}^{2}v-\partial_{t}(\kappa(s)\partial_{t}v)=0\qquad\text{in
}U.\label{MA-eq-v}%
\end{equation}

\end{lemma}

\begin{proof}
Let $w\in C_{c}^{\infty}(U)$ and $W=w\circ\Gamma$. By (\ref{MA-pullback}),
Lemma \ref{lem-MA-v} and the equation, a.e. in $B$,%
\[
\left( (v_{s}w_{s}+\kappa v_{t}w_{t})\circ\Gamma\right) u_{yy}=-u_{xy}\left(
W_{x}-\frac{u_{xy}}{u_{yy}}W_{y}\right) +\frac{\kappa}{u_{yy}}W_{y}%
=u_{xx}W_{y}-u_{xy}W_{x}.
\]
By (\ref{MA-area}) and (\ref{MA-null}) with $g=u_{x}$, $h=1$, we get $\int%
_{U}(v_{s}w_{s}+\kappa v_{t}w_{t})=\int_{B}(u_{xx}W_{y}-u_{xy}W_{x})=0$.
\end{proof}

\subsection{A lower bound for $u_{yy}$}

Let $d^{\ast}$ be the Carnot--Carath\'{e}odory metric of
$A=\operatorname{diag}(1,\kappa(s))$, with balls $B^{\ast}(P,\rho)$. Since
$\sqrt{\kappa}=f_{3,\sigma}$, Section \ref{sec-geom} applies. In particular:

\begin{itemize}
\item $d^{\ast}$ is topologically equivalent to the Euclidean metric;

\item $d^{\ast}$ is invariant under vertical translations, so $|B^{\ast
}(P+he_{2},\rho)|=|B^{\ast}(P,\rho)|$.
\end{itemize}

\begin{lemma}
\label{lem-MA-lower}For every compact $K\subset U$, $v_{t}\in L^{\infty}(K)$.
Equivalently, for every compact $K^{\prime}\subset B$ there is $c_{K^{\prime}%
}>0$ with $u_{yy}\geq c_{K^{\prime}}$ a.e. in $K^{\prime}$.
\end{lemma}

\begin{proof}
Fix $P\in U$. Choose $\rho,h_{0}>0$ such that the closure of $\bigcup_{0\leq
h\leq h_{0}}B^{\ast}(P+he_{2},\rho)$ lies in $U$. If $P$ is on the axis
$\{s=0\}$, also take $\rho\leq R_{0}$, where $R_{0}$ is as in Lemma
\ref{lem-geom-H}.

\emph{Difference quotients.} For $0<h<h_{0}$ let $w_{h}%
(s,t)=(v(s,t+h)-v(s,t))/h$. Since the coefficients of (\ref{MA-eq-v}) do not
depend on $t$, $w_{h}$ is a weak solution of (\ref{MA-eq-v}) near
$\overline{B^{\ast}(P,\rho)}$. It is nonnegative since $v(s,\cdot)$ is
increasing, and bounded there since $v$ is continuous.

\emph{Bounding $w_{h}$.} If $P$ is on the axis, Corollary \ref{cor-Lp}
applies, with $p=1$, $\nu=\frac{1}{2}$, right hand side $0$, and $k>0$
arbitrary (let $k\rightarrow0$). If $P$ is off the axis, the equation is
uniformly elliptic near $P$ and the same estimate is classical \cite[Theorem
8.17]{GiTr}. In both cases%
\[
\operatorname*{ess\,sup}_{B^{\ast}(P,\rho/2)}w_{h}\leq C_{\ast}%
%TCIMACRO{\TeXButton{dashint}{\dashint}}%
%BeginExpansion
\dashint
%EndExpansion
_{B^{\ast}(P,\rho)}w_{h},
\]
with $C_{\ast}$ independent of $h$.

\emph{Bounding the average.} Since $v(s,\cdot)$ is absolutely continuous for
a.e. $s$, Fubini, vertical translation invariance and (\ref{MA-vt-L1}) give%
\[%
%TCIMACRO{\TeXButton{dashint}{\dashint}}%
%BeginExpansion
\dashint
%EndExpansion
_{B^{\ast}(P,\rho)}w_{h}=\frac{1}{h}\int_{0}^{h}\frac{1}{|B^{\ast}(P,\rho
)|}\int_{B^{\ast}(P+\theta e_{2},\rho)}v_{t}\,ds\,dt\,d\theta\leq\frac
{|B|}{|B^{\ast}(P,\rho)|}.
\]

\emph{Conclusion.} Since $w_{h}\rightarrow v_{t}$ a.e. as $h\rightarrow0$,
$v_{t}\leq C_{\ast}|B|/|B^{\ast}(P,\rho)|$ a.e. in $B^{\ast}(P,\rho/2)$. Cover
$K$ by finitely many such balls. The second statement follows with
$K=\Gamma(K^{\prime})$, since $u_{yy}=1/(v_{t}\circ\Gamma)$ a.e.
\end{proof}

By Lemma \ref{lem-MA-lower} and (\ref{MA-v-bounds}), $\nabla v$ is locally
bounded in $U$. So $v$ is locally Lipschitz and $\Gamma$ is locally bi-Lipschitz.

\subsection{Equations for the derivatives of $v$}

\begin{lemma}
\label{lem-MA-vt}$v_{t}\in W_{A,\mathrm{loc}}^{1,2}(U)\cap L_{\mathrm{loc}%
}^{\infty}(U)$, and $v_{t}$ is a weak solution of $L_{A}v_{t}=0$ in $U$.
\end{lemma}

\begin{proof}
Let $V\Subset V^{\prime}\Subset U$ with $V+[0,h_{0}]e_{2}\subset V^{\prime} $.
For $0<h<h_{0}$ the difference quotients $w_{h}$ are weak solutions of
(\ref{MA-eq-v}) near $\overline{V}$, with $|w_{h}|\leq\Vert v_{t}%
\Vert_{L^{\infty}(V^{\prime})}$.

For $\psi\in C_{c}^{\infty}(V)$, testing with $\psi^{2}w_{h}$ and using
$|\nabla_{A}\psi|\leq|\nabla\psi|$ (as $\kappa\leq1$) gives%
\[
\int\psi^{2}|\nabla_{A}w_{h}|^{2}\leq4\Vert\nabla\psi\Vert_{\infty}^{2}\Vert
v_{t}\Vert_{L^{\infty}(V^{\prime})}^{2}|V|,
\]
uniformly in $h$. Since $w_{h}\rightarrow v_{t}$ a.e. and boundedly, a
subsequence of $\nabla_{A}w_{h}$ converges weakly in $L_{\mathrm{loc}}^{2}(V)$
to $\nabla_{A}v_{t}$. Passing to the limit in the weak formulation gives
$L_{A}v_{t}=0$ in $V$.
\end{proof}

\begin{lemma}
\label{lem-MA-vs}$v_{s}\in W_{A,\mathrm{loc}}^{1,2}(U)\cap L_{\mathrm{loc}%
}^{\infty}(U)$, and $v_{s}$ is a weak solution of%
\begin{equation}
L_{A}v_{s}=-\operatorname{div}_{A}\vec{\phi}_{1}\quad\text{in }U,\qquad
\vec{\phi}_{1}=\left( 0,-\frac{\kappa^{\prime}(s)}{\sqrt{\kappa(s)}}%
v_{t}\right) .\label{MA-eq-vs}%
\end{equation}
That is, $\int_{U}(\partial_{s}v_{s}\,w_{s}+\kappa\,\partial_{t}v_{s}%
\,w_{t})=\int_{U}\vec{\phi}_{1}\cdot\nabla_{A}w=-\int_{U}\kappa^{\prime}%
v_{t}\,w_{t}$ for all $w\in\operatorname{Lip}_{\mathrm{c}}(U)$. Moreover, for
every compact $K\subset U$,%
\begin{equation}
|\vec{\phi}_{1}(s,t)|\leq\Vert v_{t}\Vert_{L^{\infty}(K)}\,\eta_{\kappa
}(|s|)\qquad\text{a.e. in }K.\label{MA-drift-bound}%
\end{equation}

\end{lemma}

\begin{proof}
(\ref{MA-drift-bound}) is immediate from (\ref{eta kappa}) and Lemma
\ref{lem-MA-lower}.

\emph{Regularity.} As distributions, $\partial_{t}v_{s}=\partial_{s}v_{t}\in
L_{\mathrm{loc}}^{2}$ by Lemma \ref{lem-MA-vt}. By (\ref{MA-eq-v}),%
\[
\partial_{s}v_{s}=-\kappa\,\partial_{t}v_{t}=-\sqrt{\kappa}(\sqrt{\kappa
}\,\partial_{t}v_{t})\in L_{\mathrm{loc}}^{2}.
\]
Hence $\nabla_{A}v_{s}\in L_{\mathrm{loc}}^{2}$. Also $v_{s}\in
L_{\mathrm{loc}}^{\infty}$ by (\ref{MA-v-bounds}) and Lemma \ref{lem-MA-lower}.

\emph{The equation.} For $w\in C_{c}^{\infty}(U)$, test (\ref{MA-eq-v}) with
$\partial_{s}w$:%
\[
0=\int_{U}(v_{s}\,\partial_{s}w_{s}+\kappa v_{t}\,\partial_{s}w_{t})=-\int%
_{U}\partial_{s}v_{s}\,w_{s}-\int_{U}\partial_{s}(\kappa v_{t})\,w_{t}.
\]
Since $\partial_{s}(\kappa v_{t})=\kappa^{\prime}v_{t}+\kappa\partial_{t}%
v_{s}$, this is the claimed identity. Lipschitz test functions are handled by mollification.
\end{proof}

\subsection{Proof of Theorem \ref{th-MA}}

Extend $f_{3,\sigma}$ from $[-R,R]$ to an even Lipschitz function on
$\mathbb{R}$, positive on $\mathbb{R}\setminus\{0\}$. Let $U^{\prime}\Subset
U$ be a domain. In $U^{\prime}$ we apply Theorem \ref{th-geom-cont} with
$k=3$, the given $\sigma<\frac{1}{2}$, $\mathcal{A}=A$ (so $\lambda=\Lambda
=1$), and some $m>2$ with $\sigma m(m-1)<1$.

\begin{itemize}
\item By Lemma \ref{lem-MA-vt}, $v_{t}\in W_{A}^{1,2}(U^{\prime})$ solves the
homogeneous equation.

\item By Lemma \ref{lem-MA-vs}, $v_{s}\in W_{A}^{1,2}(U^{\prime})$ solves
(\ref{MA-eq-vs}) with $\phi_{0}=0$ and $\vec{\phi}_{1}\in L^{\infty}%
(U^{\prime})$. By (\ref{MA-drift-bound}), $\vec{\phi}_{1}$ satisfies
(\ref{drift Dini}) with $\eta_{K}=\Vert v_{t}\Vert_{L^{\infty}(K)}\eta
_{\kappa}$, and $\int_{0}^{1}\eta_{\kappa}(t)\frac{dt}{t}<\infty$ because
$\eta_{\kappa}(t)\leq C_{R}^{\prime}t$.
\end{itemize}

Hence $v_{t}$ and $v_{s}$ have continuous representatives $\tilde{v}_{t}$ and
$\tilde{v}_{s}$ in $U$, and $\tilde{v}_{t}\geq1/M$ by (\ref{MA-v-bounds}).
Define the continuous functions on $B$%
\[
H_{22}=\frac{1}{\tilde{v}_{t}\circ\Gamma},\qquad H_{12}=-\frac{\tilde{v}_{s}%
}{\tilde{v}_{t}}\circ\Gamma,\qquad H_{11}=\frac{\kappa(x)+H_{12}^{2}}{H_{22}}.
\]
By (\ref{MA-ab}), Lemma \ref{lem-MA-v} and the last statement of Lemma
\ref{lem-MA-Gamma}, $u_{yy}=H_{22}$ and $u_{xy}=H_{12}$ a.e. in $B$. Then
$u_{xx}=H_{11}$ a.e. by the equation. Thus the weak derivatives of
$u_{x},u_{y}\in W^{1,\infty}(B)$ agree a.e. with continuous functions. So
$u_{x},u_{y}\in C^{1}(B)$, i.e. $u\in C^{2}(B)$. \hfill$\square$

\appendix

\section{Properties of $\Phi_{m}$ and its conjugate\label{sec-app}}

\begin{lemma}
\label{lem-Phistar}Let $m>1$ and $c_{m}=m2^{m-1}$. There is $s_{m}$ such that%
\[
\Phi_{m}^{\ast}(s)\geq\exp\left( \left( \frac{\ln(s/2)}{c_{m}}\right)
^{\frac{m}{m-1}}\right) \qquad\text{for }s\geq s_{m}.
\]
Consequently:

\begin{itemize}
\item $(\Phi_{m}^{\ast})^{-1}(T)\leq2\exp(c_{m}(\ln T)^{\frac{m-1}{m}})$ for
$T\geq T_{m}$;

\item $\Phi_{m}^{\ast}(s)\geq s^{q}$ for $s$ large, for every $q<\infty$, so
$L^{\Phi_{m}^{\ast}}\subset L^{q}$ on sets of finite measure.
\end{itemize}
\end{lemma}

\begin{proof}
\emph{An upper bound for $\Phi_{m}$.} Let $t\geq e^{2^{m}}$ and $\ell=\ln t$,
so $\ell^{1/m}\geq2$. By the mean value theorem $(a+1)^{m}-a^{m}\leq
m(a+1)^{m-1}$. With $a=\ell^{1/m}$ and $a+1\leq2a$ this gives%
\[
(\ell^{1/m}+1)^{m}\leq\ell+m2^{m-1}\ell^{\frac{m-1}{m}},\qquad\text{i.e.}%
\qquad\Phi_{m}(t)\leq t\exp(c_{m}\ell^{\frac{m-1}{m}}).
\]

\emph{The lower bound for $\Phi_{m}^{*}$.} Given $s$, choose $t$ with $\ln
t=\ell:=(\ln(s/2)/c_{m})^{\frac{m}{m-1}}$. This $\ell$ is $\geq2^{m}$ for
$s\geq s_{m}$. Then $\Phi_{m}(t)\leq ts/2$, and%
\[
\Phi_{m}^{\ast}(s)\geq st-\Phi_{m}(t)\geq\frac{st}{2}\geq t.
\]

\emph{The consequences.} Inverting the increasing function $s\mapsto\exp
((\ln(s/2)/c_{m})^{m/(m-1)})$ gives the bound for $(\Phi_{m}^{\ast})^{-1}$.
The last statement holds since $\frac{m}{m-1}>1$.
\end{proof}

\bibliographystyle{amsplain}
\bibliography{library-Harnack}

\end{document}